\documentclass[12pt,reqno]{amsart}
\usepackage{graphicx,color}
\usepackage{multirow}
\usepackage{mathrsfs, cite, url}
\newtheorem{thm}{Theorem}[section]
\newtheorem{cor}[thm]{Corollary}
\newtheorem{lem}[thm]{Lemma}

\theoremstyle{definition}

\theoremstyle{example}
\newtheorem{exam}{Example}
\theoremstyle{remark}
\newtheorem{rem}[thm]{Remark}
\numberwithin{equation}{section}

\begin{document}
\setcounter{page}{1}
\title[Logarithmic integrals]{Logarithmic integrals with applications to BBP and Euler--type sums}
\author{necdet Bat{\i}r}%
\address{department of  mathematics, nev{\c{s}}ehir hbv university, nev{\c{s}}ehir, 50300 turkey}%
\email{nbatir@hotmail.com}%
\begin{abstract} For real numbers $p,q>1$  we consider the following family of integrals:
\begin{equation*}
\int_{0}^{1}\frac{(x^{q-2}+1)\log\left(x^{mq}+1\right)}{x^q+1}{\rm d}x \quad \mbox{and}\quad \int_{0}^{1}\frac{(x^{pt-2}+1)\log\left(x^t+1\right)}{x^{pt}+1}{\rm d}x.
\end{equation*}
We evaluate these integrals for all $m\in\mathbb{N}$, $q=2,3,4$ and $p=2,3$ explicitly. They recover some previously known integrals. We also compute  many integrals over the infinite interval $[0,\infty)$. Applying these results we offer many new Euler- BBP- type sums.
\end{abstract}

\maketitle
\section{Introduction}
In the recent paper \cite{19}, Sofo has evaluated a number of definite integrals over the unit interval involving logarithmic and polylogarithmic functions elegantly. For example, he proved
\begin{align}\label{e:1}
\int_{0}^{1}\frac{\log\left(x^4+1\right)}{x^2+1}{\rm d}x&=\frac{\pi}{4}\log \left(6+4\sqrt{2}\right)-2G,
\end{align}
\begin{align}\label{e:2}
\int_{0}^{1}\frac{\log\left(x^6+1\right)}{x^2+1}{\rm d}x=\frac{\pi}{2}\log 6-3G,
\end{align}
and

\begin{align}\label{e:3}
\int_{0}^{1}\frac{\log\left(x^{10}+1\right)}{x^2+1}{\rm d}x&=-\frac{\pi}{2}\log\left(\frac{72-32\sqrt{5}}{5}\right)-3G,
\end{align}
where $G$ is the Catalan constant; see also \cite{20,22, 5}. Our main aim in the present paper is to offer some general logarithmic integral formulas which cover the integrals given by (\ref{e:1}), (\ref{e:2}) and  (\ref{e:3}) as special cases. Our evaluations include many other logarithmic integrals. More precisely,  we evaluated the following families of integrals explicitly for many values of  the parameters $m,q,p$ and $t$: For $q,\,t>1$, $m,\,p\in\mathbb{N}$
\begin{equation*}
\int_{0}^{1}\frac{\left(x^{q-2}+1\right)\log\left(x^{mq}+1\right)}{x^q+1}{\rm d}x, \quad \int_{0}^{\infty}\frac{x^{q-2}\log\left(x^{mq}+1\right)}{x^{q}+1}{\rm d}x
\end{equation*}
and
\begin{equation*}
\int_{0}^{1}\frac{\left(x^{pt-2}+1\right)\log\left(x^{t}+1\right)}{x^{pt}+1}{\rm d}x.
\end{equation*}

In 1997 Bailey \textit{et all}. \cite{2}  discovered the following elegant series representation for the constant $\pi$, by using the PSLQ \textit{integer relation algorithm }
\begin{equation*}
\pi=\sum_{n=0}^{\infty}\frac{1}{16^n}\left(\frac{4}{8n+1}-\frac{2}{8n+4}-\frac{1}{8n+5}-\frac{1}{8n+6}\right).
\end{equation*}
This formula enables one to compute the $n$ th hexadecimal or binary digit of $\pi$, without computing any of the first $n-1$ digits. Since then this formula is known as BBP formula for $\pi$. More generally, formulas of the following type are known in the literature as BBP-type formulas for a given  mathematical constant $\alpha$:
\begin{equation*}
\alpha=\sum_{n=0}^{\infty}\frac{1}{b^n}\sum_{k=1}^{m}\frac{a_k}{(mn+k)^d},
\end{equation*}
where $b,d$ and $m$ are positive integers with $b\geq 2$, and $(a_1,a_2,...,a_m)\in\mathbb{Z}^m$. After the discovery of  the BBP-type formula for $\pi$  mentioned above,  many authors devoted themselves to find new BBP-type formulas for the other mathematical constants and they provided interesting formulas; see \cite{2, 4,29,3,26, 27,28}. In the literature there exist different variants of BBP-type formulas. In the very recent papers \cite{19,20,22} Sofo derived many series representations for various constants consisting of a combination  of the Euler sums and BBP-type series using the solutions of some definite logarithmic and polylogarithmic integrals. We recall some of them here.
\begin{align*}
\sum_{n=1}^{\infty}(-1)^{n+1}H_n\left(\frac{1}{6n+1}-\frac{1}{6n+3}+\frac{1}{6n+5}\right)=\frac{\pi}{2}\log 6-3G,
\end{align*}
see \cite[pg. 68]{19}
\begin{align*}
\sum_{n=1}^{\infty}(-1)^{n+1}H_n\bigg(\frac{1}{10n+1}&-\frac{1}{10n+3}+\frac{1}{10n+5}-\frac{1}{10n+7}+\frac{1}{10n+9}\bigg)\\
&=-\frac{\pi}{4}\log\left(\frac{72-32\sqrt{5}}{5}\right)+\frac{5\pi}{2}\log 2-5G;
\end{align*}
see \cite[pg. 166]{20}, and
\begin{align*}
2\sum_{n=1}^{\infty}H_n\bigg(\frac{1}{(8n+2)^3}&-\frac{1}{(8n+4)^3}+\frac{1}{(8n+6)^3}-\frac{1}{(8n+8)^3}\bigg)\\
&=-\frac{337}{512}\zeta(4)-\frac{1}{4}G^2-\frac{77}{128}\zeta(3)\log 2;
\end{align*}
see \cite[pg. 323]{22}. Our second aim in this paper is to present many new Euler- BBP--type formulae using the integrals calculated in the second section.

For our purpose, we recall some basic properties of the familiar gamma function $\Gamma$, and the digamma (or psi) function $\psi(x)=\Gamma^\prime(x)/\Gamma(x)$ ($x>0$), which will be used extensively throughout  the paper. The gamma function satisfies the reflection formula
\begin{equation}  \label{e:4}
\Gamma(s)\Gamma(1-s)=\frac{\pi}{\sin(\pi s)} \quad (s\in\mathbb{C}\backslash%
\mathbb{Z}).
\end{equation}
see \cite[p.\ 253]{13}. Other important functional equations satisfied by the
gamma function are the duplication formula
\begin{equation}  \label{e:5}
\Gamma\left(k+\frac{1}{2}\right)=\frac{(2k)!\sqrt{\pi}}{4^kk!};
\end{equation}
see \cite[p.\ 252]{13}, and  the Gauss multiplication formula for $k=2,3,4,...$ and $x>0$
\begin{equation}\label{e:6}
\Gamma(kx)=\frac{k^{kx-1/2}}{(2\pi)^{(k-1)/2}}\Gamma(x)\Gamma\left(x+\frac{1}{k}\right)\Gamma\left(x+\frac{2}{k}\right)\cdots\Gamma\left(x+\frac{k-1}{k}\right);
\end{equation}
see \cite[pg. 259]{13}. The beta function $B(\cdot,\cdot)$ is defined by
\begin{equation}\label{e:7}
B(s,t)=\int_{0}^{1}x^{s-1}(1-x)^{t-1}{\rm d}x=\frac{\Gamma(s)\Gamma(t)}{\Gamma(s+t)}\quad (\Re(s,t)>0).
\end{equation}
The digamma function $\psi$ and harmonic numbers $H_n$ are related to
\begin{equation}  \label{e:8}
\psi(n+1)=-\gamma+H_n \quad (n\in\mathbb{N}_0);
\end{equation}
see \cite[p.\ 31]{24}, where $\gamma=0.57721\cdots$ is the Euler-Mascheroni
constant. The digamma function $\psi$ satisfies a reflection formula and a
duplication formula similar to those of the gamma function:
\begin{equation}  \label{e:9}
\psi(s)-\psi(1-s)=-\pi\cot(\pi s) \quad (s\in\mathbb{C}\backslash\mathbb{Z}),
\end{equation}
and
\begin{equation*}
\psi\left(s+\frac{1}{2}\right)=2\psi(2s)-\psi(s)-2\log2 \quad (s\in\mathbb{C}%
\backslash\mathbb{Z}^-);
\end{equation*}
see \cite[p.\ 25]{24}. In particular, for $s=k\in\mathbb{N}$, we have
\begin{equation}  \label{e:10}
\psi\left(k+\frac{1}{2}\right)=2H_{2k}-H_k+\psi(1/2),
\end{equation}
where $H_n=1+\frac{1}{2}+\cdots+\frac{1}{n}$ is the n th harmonic number.
The following lemma has  a prominent role in the proofs of our main results given in the second section.
\begin{lem}\label{Lemma}
Let $x$ be a non-zero real number. Then we have
\begin{equation}\label{e:11}
\log\left[\left(\frac{1-x}{x}\right)^m+1\right]=\sum_{k=0}^{\frac{m}{2}-1}\log\big[1-2x(1-x)(\varphi_k+1)\big]-m\log x
\end{equation}
if $m\geq 2$ is an even integer, and
\begin{equation}\label{e:12}
\log\left[\left(\frac{1-x}{x}\right)^m+1\right]=\sum_{k=0}^{\frac{m-3}{2}}\log\big[1-2x(1-x)(\varphi_k+1)\big]-m\log x
\end{equation}
if $m\geq 1$ is an odd integer.  Thorough the paper the empty sum  $\sum_{k=0}^{-1}(\cdot)$ is taken to be zero. Here
\begin{equation}\label{e:13}
\varphi_k=\cos\left(\frac{(2k+1)\pi}{m}\right).
\end{equation}
\end{lem}
\begin{proof}
Let $m$ be a positive even integer. Then for any $z\in\mathbb{C}$ we have

\begin{equation*}
z^m+1=\prod_{k=0}^{\frac{m}{2}-1}\left(z-e^{\frac{\pi i(2k+1)}{m}}\right)\left(z-e^{-\frac{\pi i(2k+1)}{m}}\right).
\end{equation*}
Setting $z=\frac{1-x}{x}$ $(x\neq 0)$, and taking the logarithm of both sides,  we get, after a little simplification,
\begin{align*}
&\log\left(\left(\frac{1-x}{x}\right)^m+1\right)\\
&=\sum_{k=0}^{\frac{m}{2}-1}\log\bigg[\bigg(\frac{1-x}{x}-\cos\frac{(2k+1)\pi}{m}\bigg)^2+\sin^2\frac{(2k+1)\pi}{m}\bigg]\\
&=\sum_{k=0}^{\frac{m}{2}-1}\log\left(\frac{(1-x)^2-2x(1-x)\varphi_k+x^2}{x^2}\right)\\
&=\sum_{k=0}^{\frac{m}{2}-1}\log\left[1-2x(1-x)(\varphi_k+1)\right]-m\log x.
\end{align*}
Now if $m$ is an odd positive integer, then we have
\begin{equation*}
z^m+1=(z+1)\prod_{k=0}^{\frac{m-3}{2}}\left(z-e^{\frac{\pi i(2k+1)}{m}}\right)\left(z-e^{-\frac{\pi i(2k+1)}{m}}\right).
\end{equation*}
Now we set $z=\frac{1-x}{x}$ $(x\neq 0)$  here and then we take the logarithm of both sides to get

\begin{align*}
&\log\left(\left(\frac{1-x}{x}\right)^m+1\right)=-\log x\\
&+\sum_{k=0}^{\frac{m-3}{2}}\log\bigg[\bigg(\frac{1-x}{x}-\varphi_k\bigg)^2+\sin^2\frac{(2k+1)\pi}{m}\bigg]\\
&=-\log x+\sum_{k=0}^{\frac{m-3}{2}}\log\left(\frac{(1-x)^2-2x(1-x)\varphi_k+x^2}{x^2}\right)\\
&=\sum_{k=0}^{\frac{m-3}{2}}\log\left[1-2x(1-x)(\varphi_k+1)\right]-m\log x
\end{align*}
as required.
\end{proof}

\section{Evaluation of some families of logarithmic integrals}
In this section we calculate a large class of definite integrals over the unit interval $[0,1]$ and $[0,\infty)$. Applications of these results lead to many Euler-BBP-like formulas, as presented in the next section.
\begin{thm}\label{Theorem1}Let $q>1$ be any real number. Then we have
\begin{align}\label{e:14}
&\int_{0}^{1}\frac{\left(x^{q-2}+1\right)\log\left(x^{mq}+1\right)}{x^q+1}{\rm d }x=\frac{m}{4q}\left(\psi^\prime\left(\frac{2q-1}{2q}\right)-\psi^\prime\left(\frac{q-1}{2q}\right)\right)\notag\\
&-\frac{1}{q}\sum_{k=0}^{\frac{m}{2}-1}\sum_{p=1}^{\infty}\frac{2^p(\varphi_k+1)^p}{p}\frac{\Gamma\left(p+\frac{1}{q}\right)\Gamma\left(p+1-\frac{1}{q}\right)}{(2p)!}\notag\\
&-\frac{m\pi}{q\sin\frac{\pi}{q}}\left(\psi\left(1-\frac{1}{q}\right)-\psi(1)\right)
\end{align}
if $m\geq 2$ is an even integer, and

\begin{align}\label{e:15}
&\int_{0}^{1}\frac{\left(x^{q-2}+1\right)\log\left(x^{mq}+1\right)} {x^q+1}{\rm d }x=\frac{m}{4q}\left(\psi^\prime\left(\frac{2q-1}{2q}\right)-\psi^\prime\left(\frac{q-1}{2q}\right)\right)\notag\\
&-\frac{1}{q}\sum_{k=0}^{\frac{m-3}{2}}\sum_{p=1}^{\infty}\frac{2^p(\varphi_k+1)^p}{p}\frac{\Gamma\left(p+\frac{1}{q}\right)\Gamma\left(p+1-\frac{1}{q}\right)}{(2p)!}\notag\\
&-\frac{m\pi}{q\sin\frac{\pi}{q}}\left(\psi\left(1-\frac{1}{q}\right)-\psi(1)\right)
\end{align}
if $m\geq 1$ is an odd integer. Here $\varphi_k$ is as defined by (\ref{e:13}) and the sum $\sum_{k=0}^{-1}(\cdot)$ is taken to be vanish.
\end{thm}
\begin{proof}Let $m\geq 2$ be an even integer. Making the change of variable $x=\frac{1}{u}$ yields
\begin{align*}
\int_{0}^{1}\frac{\log\left(x^{mq}+1\right)}{x^q+1}{\rm d }x&=\int_{1}^{\infty}\frac{u^{q-2}\log\left(u^{mq}+1\right)}{u^q+1}{\rm d}u\\
&-mq\int_{1}^{\infty}\frac{u^{q-2}\log x}{u^q+1}{\rm d}u.
\end{align*}
We can rewrite this as follows

\begin{align}\label{e:16}
&\int_{0}^{1}\frac{\left(x^{q-2}+1\right)\log\left(x^{mq}+1\right)}{x^q+1}{\rm d }x=\underbrace{\int_{0}^{\infty}\frac{u^{q-2}\log\left(u^{mq}+1\right)}{u^q+1}{\rm d }u}_{A}\notag\\
&-mq\underbrace{\int_{0}^{\infty}\frac{u^{q-2}\log u}{u^q+1}{\rm d }u}_B+mq\underbrace{\int_{0}^{1}\frac{u^{q-2}\log u}{u^q+1}{\rm d }u}_C.
\end{align}
First we evaluate $A$. Making the substitution $\frac{1}{u^q+1}=t$ or $u=\left(\frac{1-t}{t}\right)^{\frac{1}{q}}$, we deduce that
\begin{align} \label{e:17}
A=\frac{1}{q}\int_{0}^{1}(1-t)^{-\frac{1}{q}}t^{\frac{1}{q}-1}\log\left(\left(\frac{1-t}{t}\right)^m+1\right){\rm d }t.
\end{align}
Now applying (\ref{e:11}) gives
\begin{align}\label{e:18}
A&=\frac{1}{q}\sum_{k=0}^{\frac{m-2}{2}}\int_{0}^{1}(1-t)^{-\frac{1}{q}}t^{\frac{1}{q}-1}\log\left(1-2t(1-t)(\varphi_k+1)\right){\rm d }t\notag\\
&-\frac{m}{q}\int_{0}^{1}(1-t)^{-\frac{1}{q}} t^{\frac{1}{q}-1}\log t{\rm d }t\notag\\
&=-\frac{1}{q}\underbrace{\sum_{k=0}^{\frac{m-2}{2}}\int_{0}^{1}(1-t)^{-\frac{1}{q}}t^{\frac{1}{q}-1}\left(\sum_{n=1}^{\infty}\frac{[2(\varphi_k+1)]^n[t(1-t)]^n}{n}\right){\rm d}t}_{A_1}\notag\\
&-\frac{m}{q}\underbrace{\int_{0}^{1}(1-t)^{-\frac{1}{q}} t^{\frac{1}{q}-1}\log t{\rm d }t}_{A_2}.
\end{align}
Since,  for any $n\in\mathbb{N}$ and $t\in (0,1)$, $\sum_{\upsilon=1}^n\frac{[2(\varphi_k+1)]^\upsilon[t(1-t)]^\upsilon}{\upsilon}\leq H_n$, the dominated convergence theorem justifies interchanging the  order of summation and integration. Thus, we have

\begin{align}\label{e:19}
A_1&=-\sum_{k=0}^{\frac{m-2}{2}}\sum_{p=1}^{\infty}\frac{2^p(\varphi_k+1)^p}{p}\int_{0}^{1}(1-t)^{p-\frac{1}{q}}t^{p+\frac{1}{q}-1}{\rm d }t.
\end{align}
From the gamma-beta functional relation (\ref{e:7}) we have that
\begin{align*}
\int_{0}^{1}(1-t)^{p-\frac{1}{q}}t^{p+\frac{1}{q}-1}{\rm d }t=\frac{\Gamma\left(p+\frac{1}{q}\right)\Gamma\left(p+1-\frac{1}{q}\right)}{(2p)!}.
\end{align*}
Using this in (\ref{e:19}) we deduce

\begin{align}\label{e:20}
A_1&=-\sum_{k=0}^{\frac{m-2}{2}}\sum_{p=1}^{\infty}\frac{2^p(\varphi_k+1)^p}{p}\frac{\Gamma\left(p+\frac{1}{q}\right)\Gamma\left(p+1-\frac{1}{q}\right)}{(2p)!}.
\end{align}
Proceeding as we did before, after a simple calculation, we see that
\begin{align*}
A_2&=-\int_{0}^{1}\frac{d}{ds}t^{s}\bigg|_{s=1/q-1}(1-t)^{-1/q}{\rm d }t=\frac{d}{ds}B(s+1,1-1/q)\bigg|_{s=1/q-1},
\end{align*}
where the interchanging the order of integration and differentiation can be justified by Lemma 2.1 from \cite{12}, and  $B(.,.)$ is the beta function. From the gamma-beta functional equation (\ref{e:7}) we get
\begin{align*}
A_2&=\frac{d}{ds}\frac{\Gamma(s+1)\Gamma\left(1-\frac{1}{q}\right)}{\Gamma\left(s+2-\frac{1}{q}\right)}\bigg|_{s=\frac{1}{q}-1}=\Gamma\left(\frac{1}{q}\right)\Gamma\left(1-\frac{1}{q}\right)(\psi(1/q)-\psi(1)).
\end{align*}
By the reflection formula for the gamma function this becomes
\begin{align}\label{e:21}
A_2=\frac{\pi}{\sin \frac{\pi}{q}}(\psi(1/q)-\psi(1)).
\end{align}
Combining (\ref{e:18}), (\ref{e:20}) and (\ref{e:21}) we arrive at
\begin{align}\label{e:22}
A&=-\frac{1}{q}\sum_{k=0}^{\frac{m-2}{2}}\sum_{p=1}^{\infty}\frac{2^p(\varphi_k+1)^p}{p}\frac{\Gamma\left(p+\frac{1}{q}\right)\Gamma\left(p+1-\frac{1}{q}\right)}{(2p)!}   \notag\\
&-\frac{m\pi}{q\sin \frac{\pi}{q}}(\psi(1/q)-\psi(1)).
\end{align}
Making the same change of variable $u=\left(\frac{1-t}{t}\right)^{\frac{1}{q}}$ again we find
\begin{equation*}
B=\frac{1}{q^2}\int_{0}^{1}\frac{\left(\frac{1-t}{t}\right)^{1-2/q}\log\left(\frac{1-t}{t}\right)}{t}\left(\frac{1-t}{t}\right)^{1/q-1}{\rm d }t.
\end{equation*}
We can simplify this to
\begin{align*}
B&= \frac{1}{q^2}\int_{0}^{1}t^{-1/q}(1-t)^{1/q-1}\log t {\rm d }t-\frac{1}{q^2}\int_{0}^{1}t^{1/q-1}(1-t)^{-1/q}\log t {\rm d }t\\
&=\frac{1}{q^2}\frac{d}{ds}\int_{0}^{1}t^{s}(1-t)^{1/q-1} {\rm d }t\bigg|_{s=-\frac{1}{q}}-\frac{1}{q^2}\frac{d}{ds}\int_{0}^{1}t^{s}(1-t)^{-1/q}{\rm d }t\bigg|_{s=\frac{1}{q}-1} .
\end{align*}
Employing the gamma-beta function identity we can write this as follows
\begin{align*}
B&=\frac{1}{q^2}\frac{d}{ds}\frac{\Gamma(s+1)\Gamma(1/q)}{\Gamma(s+1+1/q)}\bigg|_{s=-1/q}\\
&-\frac{1}{q^2}\frac{d}{ds}\frac{\Gamma(s+1)\Gamma(1-1/q)}{\Gamma(s+2-1/q)}\bigg|_{s=1/q-1}
\end{align*}
or
\begin{align*}
B&=\frac{\Gamma(1/q)\Gamma(1-1/q)}{q^2}\left[\psi(1-1/q)-\psi(1/q)\right].
\end{align*}
Using the duplication formulas for the gamma  function we find
\begin{equation}\label{e:23}
B=\frac{\pi}{q^2\sin(\pi/q)}\left[\psi(1-1/q)-\psi(1/q)\right].
\end{equation}
Finally we shall evaluate $C$. Expanding the log--function in the integral defining $C$ in (\ref{e:16}) to its Maclaurin series
we get
\begin{align*}
C&=\int_{0}^{1}u^{q-2}\log u\sum_{k=0}^{\infty}(-1)^ku^{qk}{\rm d }u.
\end{align*}
This series is uniformly convergent, so, we can interchange the order of integration and summation, and
\begin{align*}
C=\sum_{k=0}^{\infty}(-1)^k\int_{0}^{1}u^{q+qk-2}\log u{\rm d }u=\sum_{k=1}^{\infty}\frac{(-1)^{k}}{(qk-1)^2}.
\end{align*}
But since
\begin{align*}
\sum_{k=1}^{\infty}\frac{(-1)^{k}}{(qk-1)^2}=\frac{1}{4q^2}\left[\psi^\prime\left(\frac{2q-1}{2q}\right)-\psi^\prime\left(\frac{q-1}{2q}\right)\right]
\end{align*}
we can get
\begin{align}\label{e:24}
C=\frac{1}{4q^2}\left[\psi^\prime\left(\frac{2q-1}{2q}\right)-\psi^\prime\left(\frac{q-1}{2q}\right)\right].
\end{align}
Replacing the quantities given in (\ref{e:22}), (\ref{e:23}), and (\ref{e:24}) in (\ref{e:16}) we arrive at \eqref{e:14}. The proof of (\ref{e:15}) can be done  very similarly to that of (\ref{e:14}) by employing formula (\ref{e:12}).
\end{proof}
Setting $m=1$ in (\ref{e:15}) we have the following corollary:
\begin{cor}\label{Corollary1}For any real number with $q>1$ we have
\begin{align}\label{e:25}
\int_{0}^{1}\frac{\left(x^{q-2}+1\right)\log\left(x^{q}+1\right)}{x^q+1}{\rm d }x&=\frac{1}{4q}\left[\psi^\prime\left(\frac{2q-1}{2q}\right)-\psi^\prime\left(\frac{q-1}{2q}\right)\right]\notag\\
&-\frac{\pi}{q\sin\frac{\pi}{q}}\left(\psi\left(1-\frac{1}{q}\right)-\psi(1)\right).
\end{align}
\end{cor}
\begin{cor}\label{Corollary2}For all  $q>1$ real numbers we have
\begin{align*}
\int_{0}^{1}\frac{\left(x^{q-2}+1\right)\log\left(x^{2q}+1\right)}{x^q+1}{\rm d }x&=-\frac{2\pi}{q\sin(\pi/q)}\left[\psi(1-1/q)-\psi(1)\right]\\
&+\frac{1}{2q}\left[\psi^\prime\left(\frac{2q-1}{2q}\right)-\psi^\prime\left(\frac{q-1}{2q}\right)\right]\\
&-\frac{1}{q}\sum_{k=1}^{\infty}\frac{2^k}{k}\frac{\Gamma\left(k+\frac{1}{q}\right)\Gamma\left(k+1-\frac{1}{q}\right)}{(2k)!}.
\end{align*}
\end{cor}
\begin{proof}
The proof follows from (\ref{e:14}) by putting $m=2$.
\end{proof}

\begin{cor}\label{Corollary3} For all $q>1$ reals we have
\begin{align*}
\int_{0}^{1}\frac{\left(x^{q-2}+1\right)\log\left(x^{3q}+1\right)}{x^q+1}{\rm d }x&=\frac{3}{4q}\left(\psi^\prime\left(\frac{2q-1}{2q}\right)-\psi^\prime\left(\frac{q-1}{2q}\right)\right)\\
&-\frac{1}{q}\sum_{k=1}^{\infty}\frac{3^k}{k}\frac{\Gamma\left(k+\frac{1}{q}\right)\Gamma\left(k+1-\frac{1}{q}\right)}{(2k)!}\\
&-\frac{3\pi}{q\sin\frac{\pi}{q}}\left(\psi(1-1/q)-\psi(1)\right).
\end{align*}
\end{cor}
\begin{proof}
The proof follows from (\ref{e:15}) by putting $m=3$.
\end{proof}
\begin{thm}\label{Theorem2.4}
For any even integer $m\geq 2$, we have
\begin{align}\label{e:26}
&\int_0^1 \frac{\log \left(x^{2m}+1\right)}{x^2+1}{\rm d }x=\frac{m\pi}{4}\log 2-mG+\frac{\pi}{2}\sum_{k=0}^{\frac{m}{2}-1}\log\left(1+\sqrt{\frac{1-\varphi_k}{2}}\right),
\end{align}
and if $m\geq 1$ is any odd integer
\begin{align}\label{e:27}
\int_0^1 \frac{\log \left(x^{2 m}+1\right)}{x^2+1}{\rm d }x&=\frac{(m+1)\pi}{4}\log 2-mG\notag\\
&+\frac{\pi}{2}\sum_{k=0}^{\frac{m-3}{2}}\log\left(1+\sqrt{\frac{1-\varphi_k}{2}}\right),
\end{align}
where $\varphi_k=\cos\left(\frac{(2k+1)\pi}{m}\right)$ and $G$ is the Catalan constant.
\end{thm}
\begin{proof} Let $m\geq 2$ be any even integer. Then setting $q=2$ in (\ref{e:14}), using the duplication formula (\ref{e:4}), and noticing that $\psi^\prime(3/4)-\psi^\prime(1/4)=-16G$, and $\psi(1/2)-\psi(1)=-2\log 2$, we get
\begin{align}\label{e:28}
\int_0^1 \frac{\log \left(x^{2 m}+1\right)}{x^2+1}{\rm d }x&=m\pi\log 2-mG
 -\frac{1}{2}\sum_{k=0}^{\frac{m-3}{2}}\sum_{p=1}^{\infty}\frac{1}{p}\left(\frac{1+\varphi_k}{8}\right)^p\binom{2p}{p}.
\end{align}
\textit{Mathematica }can evaluate this series and it gives us

\begin{align}\label{e:29}
\sum_{p=1}^{\infty}\frac{1}{p}\left(\frac{1+\varphi_k}{8}\right)^p\binom{2p}{p}=2\log 2-2\log\left(1+\sqrt{\frac{1-\varphi_k}{2}}\right).
\end{align}
Applying this  to (\ref{e:28}) the proof follows. In the case of odd m the proof can be done similarly by setting $q=2$ in (\ref{e:15}).
\end{proof}
\begin{thm}\label{Theorem2.5}Let $m\geq 2$ be any even integer. Then
\begin{align}\label{e:30}
&\int_0^1 \frac{(x+1)\log \left(x^{3 m}+1\right)}{x^3+1}{\rm d }x=\frac{m}{12}\left(\psi^\prime(5/6)-\psi^\prime(1/3)\right)-\frac{\pi^2m}{9}\notag\\
&-\frac{\pi m\sqrt{3}}{9}\log 2-\frac{2\pi\sqrt{3}}{9}\bigg\{2\sum_{k=0}^{\frac{m}{2}-1}\log\bigg[\sin\bigg(\frac{1}{3}\sin^{-1}\bigg(\sqrt{\frac{\varphi_k+1}{2}}\bigg)\bigg)\bigg]\notag\\
&-\sum_{k=0}^{\frac{m}{2}-1}\log \bigg[2 \cos \bigg(\frac{2}{3} \sin ^{-1}\bigg(\sqrt{\frac{1+\varphi_k}{2}}\bigg)\bigg)+1\bigg]-\sum_{k=0}^{\frac{m}{2}-1}\log\left(\varphi_k+1\right)\bigg\}.
\end{align}
If $m\geq 1$ is an odd integer

\begin{align}\label{e:31}
&\int_0^1 \frac{(x+1)\log \left(x^{3 m}+1\right)}{x^3+1}{\rm d }x=\frac{m}{12}\left(\psi^\prime(5/6)-\psi^\prime(1/3)\right)-\frac{\pi^2m}{9}\notag\\
&+\frac{\pi\sqrt{3}}{3}\log 3-\frac{\pi(m-1)\sqrt{3}}{9}\log 2\notag\\
&-\frac{2\pi\sqrt{3}}{9}\bigg\{2\sum_{k=0}^{\frac{m-3}{2}}\log\bigg[\sin\bigg(\frac{1}{3}\sin^{-1}\bigg(\sqrt{\frac{\varphi_k+1}{2}}\bigg)\bigg)\bigg]\notag\\
&-\sum_{k=0}^{\frac{m-3}{2}}\log \bigg[2 \cos \bigg(\frac{2}{3} \sin ^{-1}\bigg(\sqrt{\frac{1+\varphi_k}{2}}\bigg)\bigg)+1\bigg]-\sum_{k=0}^{\frac{m-3}{2}}\log\left(1+\varphi_k\right)\bigg\}.
\end{align}
\end{thm}
\begin{proof}
Let $m\geq 2$ be an even integer. Setting $q=3$ in (\ref{e:14}) we get
\begin{align}\label{e:32}
&\int_0^1 \frac{(x+1)\log \left(x^{3 m}+1\right)}{x^3+1}{\rm d }x=\frac{m}{12}\left(\psi^\prime(5/6)-\psi^\prime(1/3)\right)\notag\\
&-\frac{2m\pi\sqrt{3}}{9}\left(\psi(2/3)-\psi(1)\right)-\frac{1}{3}\sum_{k=0}^{\frac{m}{2}-1}\sum_{p=1}^{\infty}\frac{2^p(1+\varphi_k)^p}{p}\frac{\Gamma\left(p+\frac{1}{3}\right)\Gamma\left(p+\frac{2}{3}\right)}{(2p)!}.
\end{align}
By the Gauss multiplication formula  (\ref{e:6}) for the $\Gamma-$ function for $x=p$ and $k=2$ we have
\begin{equation}\label{e:33}
\Gamma\left(p+\frac{1}{3}\right)\Gamma\left(p+\frac{2}{3}\right)=\frac{2\pi\sqrt{3}(3p)!}{3^{3p+1}p!}.
\end{equation}
We therefore can write (\ref{e:32})  as follows:
\begin{align}\label{e:34}
&\int_0^1 \frac{(x+1)\log \left(x^{3 m}+1\right)}{x^3+1}{\rm d }x=\frac{m}{12}\left(\psi^\prime(5/6)-\psi^\prime(1/3)\right)\notag\\
&-\frac{2m\pi\sqrt{3}}{9}\left(\psi(2/3)-\psi(1)\right)-\frac{2\pi\sqrt{3}}{9}\sum_{k=0}^{\frac{m}{2}-1}\sum_{p=1}^{\infty}\frac{\left(\lambda_k\right)^p}{p}\binom{3p}{p},
\end{align}
where $\lambda_k=\frac{2(\varphi_k+1)}{27}$. \textit{Mathematica} gives us, for $|x|<\frac{4}{27}$
\begin{align*}
\sum _{p=1}^{\infty } \binom{3 p}{p} x^{p-1}=\frac{2 \cos \left(\frac{1}{3} \sin ^{-1}\left(\frac{3 \sqrt{3x}}{2}\right)\right)-\sqrt{4-27 x}}{x\sqrt{4-27 x}}.
\end{align*}
If we apply  the operator $\int_{0}^{\lambda_k}(\cdot){\rm d}x$ to both sides, \textit{Mathematica} is not able to give any result for the right-hand side. But,  it is easily able to evaluate the anti-derivative of the right-hand side of the above equality, as below:

\begin{align*}
&-\log (x)+2 \log \left[\sin \left(\frac{1}{3} \sin ^{-1}\left(\frac{3 \sqrt{3 x}}{2}\right)\right)\right]\\
&-\log \left[2 \cos \left(\frac{2}{3} \sin ^{-1}\left(\frac{3 \sqrt{3 x}}{2}\right)\right)+1\right].
\end{align*}
This can be easily verified by differentiating this function. Computing the required limits  for  $x\to 0$ and $x\to\lambda_k$, this gives us  the following equality:
\begin{align}\label{e:35}
\sum _{p=1}^{\infty } \binom{3 p}{p} \frac{(\lambda_k)^{p}}{p}&=2 \log \bigg[\sin \bigg(\frac{1}{3} \sin ^{-1}\bigg(\sqrt{\frac{1+\varphi_k}{2}}\bigg)\bigg)\bigg]\notag\\
&-\log \bigg[2 \cos \bigg(\frac{2}{3} \sin ^{-1}\bigg(\sqrt{\frac{1+\varphi_k}{2}}\bigg)\bigg)+1\bigg]\notag\\
&+\log 2+3\log 3-\log\left(\varphi_k+1\right).
\end{align}
Putting this in  (\ref{e:34}), and noting that $\psi(2/3)-\psi(1)=\frac{\pi\sqrt{3}}{6}-\frac{3}{2}\log 3$, we get
\allowdisplaybreaks
\begin{align*}
&\int_0^1 \frac{(x+1)\log \left(x^{3 m}+1\right)}{x^3+1}{\rm d }x=\frac{m}{12}\left(\psi^\prime(5/6)-\psi^\prime(1/3)\right)-\frac{\pi^2m}{9}\\
&-\frac{\pi m\sqrt{3}}{9}\log 2-\frac{2\pi\sqrt{3}}{9}\bigg\{2\sum_{k=0}^{\frac{m}{2}-1}\log\bigg[\sin\bigg(\frac{1}{3}\sin^{-1}\bigg(\sqrt{\frac{\varphi_k+1}{2}}\bigg)\bigg)\bigg]\\
&-\sum_{k=0}^{\frac{m}{2}-1}\log\bigg[2\cos\left(\frac{(2k+1)\pi}{3m}\right)+1\bigg]-\sum_{k=0}^{\frac{m}{2}-1}\log\left(\varphi_k+1\right)\bigg\}.
\end{align*}
which is the desired result (\ref{e:30}).  The proof of (\ref{e:31}) can be done similarly.
\end{proof}

\begin{thm}\label{Theorem2.6}Let $m\geq 2$ be any even integer. Then for  $m\geq 2$ even integers
\begin{align}\label{e:36}
&\int_0^1 \frac{(x^2+1)\log \left(x^{4 m}+1\right)}{x^4+1}{\rm d }x=\frac{m}{16}\left(\psi^\prime(7/8)-\psi^\prime(3/8)\right)-\frac{\pi^2m\sqrt{2}}{8}\notag\\
&+\frac{\pi m\sqrt{2}}{4}\log 2+\frac{\pi\sqrt{2}}{4}\bigg\{\sum_{k=0}^{\frac{m}{2}-1}\log\left(1+\sqrt{\frac{1-\varphi_k}{2}}\thinspace\right)\notag\\
&+2\sum_{k=0}^{\frac{m}{2}-1}\log\left(\sqrt{2}+\sqrt{1+\sqrt{\frac{1-\varphi_k}{2}}}\thinspace\right)\bigg\},
\end{align}
and for $m\geq1$ odd integers
\begin{align}\label{e:37}
&\int_0^1 \frac{(x^2+1)\log \left(x^{4 m}+1\right)}{x^4+1}{\rm d }x=\frac{m}{16}\left(\psi^\prime(7/8)-\psi^\prime(3/8)\right)-\frac{\pi^2m\sqrt{2}}{8}\notag\\
&+\frac{\pi (m+2)\sqrt{2}}{4}\log 2+\frac{\pi\sqrt{2}}{4}\bigg\{\sum_{k=0}^{\frac{m-3}{2}}\log\left(1+\sqrt{\frac{1-\varphi_k}{2}}\thinspace\right)\notag\\
&+2\sum_{k=0}^{\frac{m-3}{2}}\log\left(\sqrt{2}+\sqrt{1+\sqrt{\frac{1-\varphi_k}{2}}}\thinspace\right)\bigg\},
\end{align}
where $\varphi_k=\cos\left(\frac{(2k+1)\pi}{m}\right)$.
\end{thm}
\begin{proof}
Let $m$ be an even integer. Setting $q=4$ in (\ref{e:14}), we can get
\begin{align*}
&\int_0^1 \frac{(x^2+1)\log \left(x^{4 m}+1\right)}{x^4+1}{\rm d }x=\frac{m}{16}(\psi^\prime(7/8)-\psi^\prime(3/8))\\
&-\frac{m\pi\sqrt{2}}{4}(\psi(3/4)-\psi(1))-\frac{1}{4}\sum_{k=0}^{\frac{m}{2}-1}\sum_{p=1}^{\infty}\frac{2^p\left(\varphi_k+1\right)^p}{p}\frac{\Gamma\left(p+\frac{1}{4}\right)\Gamma\left(p+\frac{3}{4}\right)}{(2p)!}.
\end{align*}
Using (\ref{e:5}) and (\ref{e:6}) for $k=4$ and $x=p$ we find that

\begin{equation}\label{e:38}
\Gamma\left(p+\frac{1}{4}\right)\Gamma\left(p+\frac{3}{4}\right)=\frac{\pi\sqrt{2}(4p)!}{(2p)!4^{3p}}.
\end{equation}

Applying this and $\psi(3/4)-\psi(1)=\frac{\pi^2}{2}-3\log 2$,  we arrive at

\begin{align}\label{e:39}
&\int_0^1 \frac{(x^2+1)\log \left(x^{4 m}+1\right)}{x^4+1}{\rm d }x=\frac{m}{16}(\psi^\prime(7/8)-\psi^\prime(3/8))\notag\\
&-\frac{\pi^2m\pi\sqrt{2}}{8}+\frac{3\pi m\sqrt{2}}{4}\log 2-\frac{\pi\sqrt{2}}{4}\sum_{k=0}^{\frac{m}{2}-1}\sum_{p=1}^{\infty}\frac{1}{p}\left(\frac{\varphi_k+1}{32}\right)^p\binom{4p}{2p}.
\end{align}
By \cite[Theorem 3.1]{7} we have for $|x|\leq \frac{1}{16}$
\begin{align*}
&\sum_{n=1}^{\infty}\frac{1}{n}\binom{4n}{2n}x^n\\
&=4\log 2-\log\left(1+\sqrt{1-16x}\right)-2\log\left(\sqrt{2}+\sqrt{1+\sqrt{1-16x}}\thinspace\right);
\end{align*}
see also \cite[Eq. (16)]{9} and \cite[Proposition 2]{17}. For $x=\frac{\varphi_k+1}{32}$  this gives

\begin{align*}
&\sum_{p=1}^{\infty}\frac{1}{p}\left(\frac{\varphi_k+1}{32}\right)^p\binom{4p}{2p}\\
&=4\log 2-\log\left(1+\sqrt{\frac{1-\varphi_k}{2}}\thinspace\right)-2\log\left(\sqrt{2}+\sqrt{1+\sqrt{\frac{1-\varphi_k}{2}}}\thinspace\right).
\end{align*}
Substituting this in (\ref{e:39}) and making some simplifications, we get the desired result (\ref{e:36}). The proof of (\ref{e:37}) can be done in a similar way to (\ref{e:36}) by putting $q=4$ in (\ref{e:15}).
\end{proof}

\begin{thm}\label{Theorem2.7} For all $q>\frac{1}{2}$ we have
\begin{align*}
&\int_{0}^{1}\frac{\left(x^{2q-2}+1\right)\log\left(x^q+1\right)}{x^{2q}+1}{\rm d }x=\frac{1}{16q}\left[\psi^\prime\left(\frac{4q-1}{4q}\right)-\psi^\prime\left(\frac{2q-1}{4q}\right)\right]\notag\\
&-\frac{1}{q}\sum_{n=0}^{\infty}\frac{2^n\Gamma\left(n+2-\frac{1}{q}\right)\Gamma\left(n+\frac{1}{q}\right)}{(2n+1)!}\left[\psi\left(n+2-\frac{1}{q}\right)-\psi(2n+2)\right].
\end{align*}
\end{thm}
\begin{proof}Making the substitution $x=\frac{1}{u}$ yields
\allowdisplaybreaks
\begin{align*}
&\int_{0}^{1}\frac{\log\left(x^q+1\right)}{x^{2q}+1}{\rm d }x=\int_{1}^{\infty}\frac{u^{2q-2}\log\left(u^q+1\right)}{u^{2q}+1}{\rm d }u-q\int_{1}^{\infty}\frac{u^{2q-2}\log u}{u^{2q}+1}{\rm d }u\\
&=\int_{0}^{\infty}\frac{u^{2q-2}\log\left(u^q+1\right)}{u^{2q}+1}{\rm d }u-q\int_{0}^{\infty}\frac{u^{2q-2}\log u}{u^{2q}+1}{\rm d }u\\
&-\int_{0}^{1}\frac{u^{2q-2}\log\left(u^q+1\right)}{u^{2q}+1}{\rm d }u+q\int_{0}^{1}\frac{u^{2q-2}\log u}{u^{2q}+1}{\rm d }u.
\end{align*}
This can be rewritten as follows:
\begin{align}\label{e:40}
&\int_{0}^{1}\frac{\left(x^{2q-2}+1\right)\log\left(x^q+1\right)}{x^{2q}+1}{\rm d }x=\int_{0}^{\infty}\frac{u^{2q-2}\log\left(u^q+1\right)}{u^{2q}+1}{\rm d }u\notag\\
&-q\int_{0}^{\infty}\frac{u^{2q-2}\log u}{u^{2q}+1}{\rm d }u+q\int_{0}^{1}\frac{u^{2q-2}\log u}{u^{2q}+1}{\rm d }u.
\end{align}
By the substitution $u=\frac{1}{u^q+1}$  we, after a little simplification, conclude that
\begin{align}\label{e:41}
\int_{0}^{\infty}\frac{u^{2q-2}\log\left(u^q+1\right)}{u^{2q}+1}{\rm d }u&-q\int_{0}^{\infty}\frac{u^{2q-2}\log u}{u^{2q}+1}{\rm d }u\notag\\
&=-\frac{1}{q}\int_{0}^{1}\frac{\left(\frac{t}{1-t}\right)^{1-1/q}\log t}{2t^2-2t+1}{\rm d }t.
\end{align}
Using the series expansion
\begin{align*}
\frac{1}{2t^2-2t+1}=\sum_{n=0}^{\infty}2^nt^n(1-t)^n,
\end{align*}
one easily see that
\begin{align*}
\int_{0}^{1}\frac{\left(\frac{t}{1-t}\right)^{1-1/q}\log t}{2t^2-2t+1}{\rm d }t=\int_{0}^{1}\left(\frac{t}{1-t}\right)^{1-1/q}\log t\sum_{n=0}^{\infty}2^nt^n(1-t)^ndt.
\end{align*}
This series is uniformly convergent, so we can interchange the order of integration and series, and
\begin{align*}
\int_{0}^{1}\frac{\left(\frac{t}{1-t}\right)^{1-1/q}\log t}{2t^2-2t+1}{\rm d }t=\sum_{n=0}^{\infty}2^n\int_{0}^{1}t^{n+1-1/q}(1-t)^{n-1+1/q}\log t{\rm d }t.
\end{align*}
By Lemma 2.1 from \cite{12} this becomes
\begin{align*}
\int_{0}^{1}\frac{\left(\frac{t}{1-t}\right)^{1-1/q}\log t}{2t^2-2t+1}{\rm d }t=\sum_{n=0}^{\infty}2^n\frac{d}{ds}\int_{0}^{1}t^{s}(1-t)^{n-1+1/q}{\rm d}t\bigg|_{s=n+1-1/q}.
\end{align*}
Using the relation (\ref{e:7}) one obtains

\begin{align}\label{e:42}
&\int_{0}^{1}\frac{\left(\frac{t}{1-t}\right)^{1-1/q}\log t}{2t^2-2t+1}{\rm d }t=\sum_{n=0}^{\infty}\frac{d}{ds}\frac{2^n\Gamma(s+1)\Gamma\left(n+\frac{1}{q}\right)}{\Gamma\left(s+n+1+\frac{1}{q}\right)}\bigg|_{s=n+1-1/q}\notag\\
&=\sum_{n=0}^{\infty}\frac{2^n\Gamma\left(n+2-\frac{1}{q}\right)\Gamma\left(n+\frac{1}{q}\right)}{(2n+1)!}\left[\psi\left(n+2-\frac{1}{q}\right)-\psi(2n+2)\right].
\end{align}
Clearly, we have
\begin{align*}
\int_{0}^{1}\frac{u^{2q-2}\log u}{u^{2q}+1}{\rm d }u&=\int_{0}^{1}u^{2q-2}\log u{\rm d }u\sum_{n=0}^{\infty}(-1)^nu^{2nk}{\rm d }u\\
&=\sum_{n=0}^{\infty}(-1)^n\int_{0}^{1}u^{2q+2nk-2}\log u{\rm d }u.
\end{align*}
It is an easy exercise to evaluate this  integral and we find
\begin{align}\label{e:43}
\int_{0}^{1}\frac{u^{2q-2}\log u}{u^{2q}+1}{\rm d }u&=\sum_{n=0}^{\infty}\frac{(-1)^{n+1}}{(2qn+2q-1)^2}\notag\\
&=\frac{1}{16q^2}\left[\psi^\prime\left(\frac{4q-1}{4q}\right)-\psi^\prime\left(\frac{2q-1}{4q}\right)\right].
\end{align}

So the desired result comes from combining the equations given in (\ref{e:40}), (\ref{e:41}), (\ref{e:42}) and (\ref{e:43}).
\end{proof}
\begin{thm}\label{Theorem2.8} For all $q>\frac{1}{2}$ we have
\begin{align*}
&\int_{0}^{1}\frac{\left(x^{3q-2}+1\right)\log\left(x^q+1\right)}{x^{3q}+1}{\rm d }x=\frac{1}{36q}\left[\psi ^\prime\left(\frac{6 q-1}{6 q}\right)-\psi ^\prime\left(\frac{3 q-1}{6 q}\right)\right]\notag\\
&+\frac{1}{q}\sum_{n=0}^{\infty}\frac{3^n\Gamma\left(n+3-\frac{1}{q}\right)\Gamma\left(n+\frac{1}{q}\right)}{(2n+2)!}\left[\psi(2n+3)-\psi\left(n+3-\frac{1}{q}\right)\right].
\end{align*}
\end{thm}
\begin{proof}
The proof can be carried out quite similarly to that of the proof of Theorem \ref{Theorem2.7}.
\end{proof}
\begin{thm}\label{Theorem2.9}Let $m\geq 1$ be any odd integer. Then we have the following identity
\begin{equation*}
\int_0^1 \frac{\log \left(x^m+1\right)}{x^2+1} {\rm d }x=\frac{\pi m}{8}\log 2-\frac{1}{2} \sum _{k=0}^{\frac{m-3}{2}} \sum _{n=0}^{\infty } \frac{2^n \sum _{p=1}^n \frac{(\varphi_k+1)^p}{p}}{(2 n+1) \binom{2 n}{n}},
\end{equation*}
where $\varphi_k=\cos\frac{(2k+1)\pi}{m}$.
\end{thm}
\begin{proof}Let us denote the integral on the left-hand side of this equality by $I$. Then making the change of variable $x=\frac{1}{t}$ and splitting the integral it follows that
\begin{equation*}
I=\int_0^\infty \frac{\log \left(t^m+1\right)}{t^2+1} {\rm d }t-I-m\int_1^\infty \frac{\log t}{t^2+1} {\rm d }t
\end{equation*}
or
\begin{equation}\label{e:44}
I=\frac{1}{2}\underbrace{\int_0^\infty \frac{\log \left(t^m+1\right)}{t^2+1} {\rm d }t}_A-\frac{m}{2}\underbrace{\int_1^\infty \frac{\log t}{t^2+1} {\rm d }t}_B.
\end{equation}
Substituting $t=\frac{1-u}{u}$ one can get
\begin{equation*}
A=\int_0^1 \frac{\log \left(\left(\frac{1-u}{u}\right)^m+1\right)}{1-2u+2u^2} {\rm d }u.
\end{equation*}
Employing Lemma 1 and expanding the denominator into its power series, we deduce that
\begin{align*}
A=\int_{0}^{1}\sum_{n=0}^{\infty}2^nu^n(1-u)^n\bigg\{\sum_{k=0}^{\frac{m-3}{2}}\log\left(1-2u(1-u)(\varphi_k+1)\right)-m\log u\bigg\}{\rm d }u.
\end{align*}
Expanding the logarithmic function  we find that
\begin{align*}
A&=-\sum_{k=0}^{\frac{m-3}{2}}\int_{0}^{1}\sum_{n=0}^{\infty}2^nu^n(1-u)^n\sum_{n=0}^{\infty}\frac{2^{n+1}u^{n+1}(1-u)^{n+1}(\varphi_k+1)^{n+1}}{n+1}{\rm d}u\\
&-m\int_{0}^{1}\sum_{n=0}^{\infty}2^nu^n(1-u)^n\log u {\rm d }u.
\end{align*}
Using the Cauchy product formula for the two convergent series, and inverting the order of integration and summation, it easily follows that
\begin{align*}
A&=-\sum_{k=0}^{\frac{m-3}{2}}\sum_{n=0}^{\infty}\sum_{p={\rm d }0}^{n}\frac{2^{n+1}(\varphi_k+1)^{p+1}}{p+1}\int_{0}^{1}u^{n+1}(1-u)^{n+1}{\rm d }u\\
&-m\sum_{n=0}^{\infty}2^n\int_{0}^{1}u^n(1-u)^n\log u {\rm d }u.
\end{align*}
Using the gamma-beta functional equation immediately yields
\begin{align*}
A=-\sum_{k=0}^{\frac{m-3}{2}}\sum_{n=1}^{\infty}\frac{2^n\sum_{p=1}^{n}\frac{(\varphi_k+1)^p}{p}}{(2n+1)\binom{2n}{n}}+m\sum_{n=0}^{\infty}\frac{2^n(H_{2n+1}-H_n)}{(2n+1)\binom{2n}{n}},
\end{align*}
from which the proof follows from (\ref{e:44}), Example 12 and the fact $B=\int_{1}^{\infty}\frac{\log x{\rm d}x}{x^2+1}=G$.
\end{proof}
\begin{thm}\label{Theorem2.10}
For even integers $m\geq 2$ we have
\begin{equation}\label{e:45}
\int_{0}^{\infty}\frac{\log\left(x^{2m}+1\right)}{x^2+1}{\rm d }x=\frac{\pi m}{2}\log2+\pi\sum_{k=0}^{\frac{m-2}{2}}\log\left(1+\sqrt{\frac{1-\varphi_k}{2}}\right).
\end{equation}
Similarly, for odd integers $m\geq 1$
\begin{equation}\label{e:46}
\int_{0}^{\infty}\frac{\log\left(x^{2m}+1\right)}{x^2+1}{\rm d }x=\frac{\pi (m+1)}{2}\log2+\pi\sum_{k=0}^{\frac{m-3}{2}}\log\left(1+\sqrt{\frac{1-\varphi_k}{2}}\right).
\end{equation}
\end{thm}
\begin{proof}We give a proof for only (\ref{e:45}). The proof of (\ref{e:46}) is similar.
Letting $q=2$ in (\ref{e:18}), we get
\begin{equation*}
\int_{0}^{\infty}\frac{\log\left(x^{2m}+1\right)}{x^2+1}{\rm d }x=\pi m\log2-\frac{\pi}{2}\sum_{k=0}^{\frac{m-2}{2}}\sum_{p=1}^{\infty}\frac{1}{p}\left(\frac{\varphi_k+1}{p}\right)^p\binom{2p}{p}.
\end{equation*}
Now the proof follows from (\ref{e:29}).
\end{proof}
\begin{thm}\label{Theorem2.11}We have
\begin{align}\label{e:47}
&\int_0^{\infty } \frac{x \log \left(x^{3m}+1\right)}{x^3+1}{\rm d }x=\frac{\pi^2m}{9}-\frac{\pi m\sqrt{3}}{9}\log 2\notag\\
&+\frac{2\pi\sqrt{3}}{9}\sum _{k=0}^{\frac{m-2}{2}} \bigg\{\log (\varphi_k+1)-2 \log \bigg[\sin \left(\frac{1}{3} \sin ^{-1}\sqrt{\frac{\varphi_k+1}{2}}\right)\bigg]\notag\\
&+\log \bigg[2 \cos \left(\frac{2}{3} \sin ^{-1}\sqrt{\frac{\varphi_k+1}{2}}\right)+1\bigg]\bigg\}
\end{align}
for all even integers $m\geq 2$, and
\begin{align}\label{e:48}
&\int_0^{\infty } \frac{x \log \left(x^{3m}+1\right)}{x^3+1}{\rm d }x=\frac{\pi^2m}{9}+\frac{\pi \sqrt{3}}{3}\log 3-\frac{\pi(m-1)\sqrt{3}}{9}\log 2\notag\\
&+\frac{2\pi\sqrt{3}}{9}\sum _{k=0}^{\frac{m-3}{2}} \bigg\{\log(\varphi_k+1)-2 \log \bigg[\sin \left(\frac{1}{3} \sin ^{-1}\sqrt{\frac{\varphi_k+1}{2}}\right)\bigg]\notag\\
&+\log \bigg[2 \cos \left(\frac{2}{3} \sin ^{-1}\sqrt{\frac{\varphi_k+1}{2}}\right)+1\bigg]\bigg\}
\end{align}
for all odd integers $m\geq 1$.
\end{thm}
\begin{proof}Let $m\geq 2$ be an even integer. The proof of (\ref{e:47}) follows from  setting $q=3$ in (\ref{e:16}), using (\ref{e:22}), and  Lemma \ref{Lemma} and noting that $\psi(1/3)-\psi(1)=-\frac{\pi\sqrt{3}}{6}-\frac{3}{2}\log 3$. For briefness we omit the details. The proof of (\ref{e:48}) can be carried out in a similar fashion.
\end{proof}
\section{Euler- BBP - type sums}
In this section we collect some basic theorems regarding Euler-BBP type series. The first theorem generalizes some results obtained by Sofo \cite{19}.
\begin{thm}\label{Theorem3.1} Let $m\geq 1$ be any odd integer. Then
\begin{align*}
\sum_{n=1}^{\infty}(-1)^{n+1}H_n\sum_{j=0}^{m-1}\frac{(-1)^j}{2mn+2j+1}&=\frac{(m+1)\pi}{4}\log 2-mG\\
&+\frac{\pi}{2}\sum_{k=0}^{\frac{m-3}{2}}\log\left(1+\sqrt{\frac{1-\varphi_k}{2}}\right),
\end{align*}
where $\varphi_k$ is as defined by (\ref{e:13}).
\end{thm}
\begin{proof}
The proof follows from (\ref{e:27}) and Theorem 1 from \cite{19}.
\end{proof}

\begin{thm}\label{Theorem3.2}Let $m\geq 1$ be any odd integer. Then
\begin{align*}
&\sum_{n=1}^{\infty}(-1)^{n+1}H_n\sum_{p=0}^{m-1}\bigg[\frac{(-1)^p}{3mn+3p+1}+\frac{(-1)^p}{3mn+3p+2}\bigg]\\
&=\frac{m}{12}\left(\psi^\prime(5/6)-\psi^\prime(1/3)\right)-\frac{\pi^2m}{9}+\frac{\pi\sqrt{3}}{3}\log 3\\
&-\frac{\pi(m-1)\sqrt{3}}{9}\log 2-\frac{2\pi\sqrt{3}}{9}\bigg\{2\sum_{k=0}^{\frac{m-3}{2}}\log\left[\sin\left(\frac{1}{3}\sin^{-1}\sqrt{\frac{\varphi_k+1}{2}}\right)\right]\notag\\
&-\sum_{k=0}^{\frac{m-3}{2}}\log \left[2 \cos \left(\frac{2}{3} \sin ^{-1}\sqrt{\frac{1+\varphi_k}{2}}\right)+1\right]-\sum_{k=0}^{\frac{m-3}{2}}\log\left(1+\varphi_k\right)\bigg\}.
\end{align*}
\end{thm}
\begin{proof}
We have
\begin{equation*}
\frac{\log\left(x^{3m}+1\right)}{x^3+1}=\frac{1-\left(-x^3\right)^{m}}{1-\left(-x^3\right)}\cdot\frac{\log\left(x^{3m}+1\right)}{1-\left(-x^3\right)^{m}}.
\end{equation*}
By the Cauchy product formula for two convergent series this leads to
\begin{align*}
\frac{\log\left(x^{3m}+1\right)}{x^3+1}&=\frac{1-\left(-x^3\right)^{m}}{1+x^3}\cdot\frac{\log\left(x^{3m}+1\right)}{1-\left(-x^3\right)^{m}}\\
&=\sum_{p=0}^{m-1}(-1)^px^{3p}\sum_{n=0}^{\infty}\sum_{k=0}^{n}\frac{(-1)^{n}x^{3(m-1)(n+1)+3(n+1)}}{k+1}\\
&=\sum_{n=1}^{\infty}(-1)^{n+1}H_n\sum_{p=0}^{m-1}(-1)^px^{3(m-1)n+3n+3p}.
\end{align*}
Multiplying both sides by $x+1$ and then integrating the resulting equality  over $(0,1)$ the conclusion follows from (\ref{e:47}).
\end{proof}
The proof of the next  two theorems can be done similarly using (\ref{e:37}) and Corollary  \ref{Corollary1}. We omit them.
\begin{thm}\label{Theorem3.3} Let $m\geq 1$ be any positive odd integer. Then
\begin{align*}
&\sum_{n=1}^{\infty}(-1)^{n+1}H_n\sum_{p=0}^{m-1}\bigg[\frac{(-1)^p}{4mn+4p+3}+\frac{(-1)^p}{4mn+4p+1}\bigg]\\
&=\frac{m}{16}\left(\psi^\prime(7/8)-\psi^\prime(3/8)\right)-\frac{\pi^2m\sqrt{2}}{8}\notag\\
&+\frac{\pi (m+2)\sqrt{2}}{4}\log 2+\frac{\pi\sqrt{2}}{4}\bigg\{\sum_{k=0}^{\frac{m-3}{2}}\log\left(1+\sqrt{\frac{1-\varphi_k}{2}}\thinspace\right)\notag\\
&+2\sum_{k=0}^{\frac{m-3}{2}}\log\left(\sqrt{2}+\sqrt{1+\sqrt{\frac{1-\varphi_k}{2}}}\thinspace\right)\bigg\}.
\end{align*}
\end{thm}
\begin{thm}\label{Theorem3.4}Let $q>1$ be any real number. Then
\begin{align*}
&\sum_{n=1}^{\infty}(-1)^{n+1}H_n\left[\frac{1}{qn+q-1}+\frac{1}{qn+1}\right]\\
&=\frac{1}{4q}\left(\psi^\prime\left(\frac{2q-1}{2q}\right)-\psi^\prime\left(\frac{q-1}{2q}\right)\right)-\frac{\pi}{q\sin\frac{\pi}{q}}\left(\psi\left(1-\frac{1}{q}\right)-\psi(1)\right).
\end{align*}
\end{thm}

\section{Examples}
In this section we provide some examples by taking particular values for the quantities  $q$ and $m$ from Theorem \ref{Theorem1} and Theorems \ref{Theorem2.4}--\ref{Theorem2.11} and  Theorems \ref{Theorem3.1}--\ref{Theorem3.4}. We also give some harmonic sum evaluations.
\begin{exam}\label{Example1} Setting $m=1$ in (\ref{e:27}) and noticing that $\psi(1/2)-\psi(1)=-2\log 2$, $\psi^\prime(1/4)=\pi^2+8G$ and $\psi^\prime(3/4)=\pi^2-8G$ we arrive at the following formula
\allowdisplaybreaks
\begin{align*}
\int_{0}^{1}\frac{\log\left(x^2+1\right)}{x^2+1}{\rm d }x&=\frac{\pi}{2}\log 2-G.
\end{align*}
\end{exam}
\begin{exam}\label{Example2}Setting $m=1$ in (\ref{e:31}) and taking into account $\psi(2/3)-\psi(1)=\frac{\pi\sqrt{3}\log 3}{3}-\frac{\pi^2}{9}$, we get
\begin{align*}
\int_{0}^{1}\frac{(x+1)\log\left(x^3+1\right)}{x^3+1}{\rm d }x=\frac{\pi  \sqrt{3} \log 3}{3}-\frac{\pi ^2}{9}+\frac{1}{12} \left[\psi ^{\prime}\left(5/6\right)-\psi ^{\prime}\left(1/3\right)\right].
\end{align*}
\end{exam}

\begin{exam}\label{Example3}If we put $m=1$ in (\ref{e:37}) we obtain
\begin{align*}
\int_{0}^{1}\frac{(x^2+1)\log\left(x^4+1\right)}{x^4+1}{\rm d }x&=\frac{1}{16}\left(\psi^\prime(7/8)-\psi^\prime(3/8)\right)\\
&-\frac{\pi^2\sqrt{2}}{8}+\frac{3\pi\sqrt{2}\log 2}{4}.
\end{align*}
\end{exam}

\begin{exam}\label{Example4} We have the following formula
\begin{align*}
\int_{0}^{1}\frac{(x^2+x)\log\left(x^5+1\right)}{x^5+1}{\rm d }x&=\frac{1}{20}\left(\psi^\prime\left(\frac{4}{5}\right)-\psi^\prime\left(\frac{3}{10}\right)\right)\\
&-\frac{4\pi}{5\sqrt{10+2\sqrt{5}}}(\psi(3/5)-\psi(1)).
\end{align*}
\end{exam}
\begin{proof}
Setting $q=5/2$ and $m=1$ in (\ref{e:15}) we get  
\begin{align}\label{e:49}
\int_{0}^{1}\frac{(\sqrt{x}+1)\log\left(x^{5/2}+1\right)}{x^{5/2}+1}{\rm d }x&=\frac{1}{10}\left(\psi^\prime(4/5)-\psi^\prime(3/10)\right)\notag\\
&-\frac{2\pi}{5\sin(2\pi/5)}(\psi(3/5)-\psi(1)).
\end{align}
Making  the change of variable $x=u^2$  and then putting $u=x$ here, we get
\begin{align}\label{e:50}
\int_{0}^{1}\frac{(\sqrt{x}+1)\log\left(x^{5/2}+1\right)}{x^{5/2}+1}{\rm d }x&=2\int_{0}^{1}\frac{(x^2+x)\log\left(x^5+1\right)}{x^5+1}{\rm d}x.
\end{align}
Combining the equations (\ref{e:49}) and (\ref{e:50}) and the fact $\sin(2\pi/5)=\frac{1}{4}\sqrt{10+2\sqrt{5}}$ the desired result follows.
\end{proof}
\begin{exam}\label{Example5}Putting $m=2$ in (\ref{e:26}) and using $\varphi_0=0$, we get the following formula:
\begin{align*}
\int_{0}^{1}\frac{\log\left(x^4+1\right)}{x^2+1}{\rm d }x&=\frac{\pi}{2}\log \left(2+\sqrt{2}\right)-2G
\end{align*}
\end{exam}

\begin{exam}\label{Example6} Putting $m=2$ in (\ref{e:30}) and using $\sin\frac{\pi}{12}=\frac{\sqrt{6}-\sqrt{2}}{4}$ we have
\begin{align*}
&\int_{0}^{1}\frac{(x+1)\log\left(x^6+1\right)}{x^3+1}{\rm d }x=\frac{1}{6}\left(\psi^\prime(5/6)-\psi^\prime(1/3)\right)-\frac{2\pi^2}{9}\\
&-\frac{2\pi\sqrt{3}}{9}\log\bigg(\frac{5-3\sqrt{3}}{4}\bigg).
\end{align*}
\end{exam}

\begin{exam}\label{Example7} Putting $m=2$ in (\ref{e:36}) we can have
\begin{align*}
&\int_{0}^{1}\frac{(x^2+1)\log\left(x^8+1\right)}{x^4+1}{\rm d }x=\frac{1}{8}\left(\psi^\prime(7/8)-\psi^\prime(3/8)\right)-\frac{\pi^2\sqrt{2}}{4}\\
&+\frac{\pi\sqrt{2}}{4} \log\left(14+8\sqrt{2}+4\sqrt{20+14\sqrt{2}}\right).
\end{align*}
\end{exam}

\begin{exam}\label{Example8}Setting $m=3$ in (\ref{e:27}) we get
\begin{align*}
\int_{0}^{1}\frac{\log\left(x^6+1\right)}{x^2+1}{\rm d }x=\frac{\pi}{2}\log 6-3G.
\end{align*}
\end{exam}
\begin{exam}\label{Example9} For $m=3$ in (\ref{e:31}) we have
\begin{align*}
&\int_{0}^{1}\frac{(x+1)\log\left(x^9+1\right)}{x^3+1}{\rm d }x=\frac{5\pi\sqrt{3}\log 3}{9}+\frac{1}{4}\left(\psi^\prime(5/6)-\psi^\prime(1/3)\right)\\
&-\frac{\pi^2}{3}-\frac{4\pi\sqrt{3}}{9}\log 2-\frac{2\pi\sqrt{3}}{9}\bigg[2 \log \left(\sin \left(\frac{\pi }{9}\right)\right)-\log \left(1+2 \cos \left(\frac{2 \pi }{9}\right)\right)\bigg].
\end{align*}
\end{exam}

\begin{exam}\label{Example10}For $m=3$ in (\ref{e:37}) we have
\begin{align*}
\int_{0}^{1}\frac{\left(x^2+1\right)\log\left(x^{12}+1\right)}{x^4+1}{\rm d }x&=\frac{3}{16}\left(\psi^\prime(7/8)-\psi^\prime(3/8)\right)-\frac{3\pi^2\sqrt{2}}{8}\\
&+\frac{\pi\sqrt{2}}{4}\log 24+\frac{\pi\sqrt{2}}{2}\log\left(2+\sqrt{3}\right).
\end{align*}
\end{exam}
\begin{exam}\label{Example11} We have
\begin{align*}
\int_{0}^{1}\frac{(x^2+1)\log\left(x^{2}+1\right)}{x^4+1}{\rm d }x&=\frac{1}{32} \left(\psi ^\prime\left(\frac{7}{8}\right)-\psi ^\prime\left(\frac{3}{8}\right)\right)\\
&-\frac{\pi ^2 \sqrt{2}}{16}+\frac{\pi  \sqrt{2}}{4}  \log \left(2+\sqrt{2}\right).
\end{align*}
\end{exam}
\begin{proof} Putting $q=2$ in Theorem \ref{Theorem2.7}, we get
\begin{align}\label{e:51}
&\int_{0}^{1}\frac{\left(x^2+1\right)\log \left(x^2+1\right)}{x^4+1}{\rm d }x=\frac{1}{32}\left(\psi^{\prime}\left(\frac{7}{8}\right)-\psi^{\prime}\left(\frac{3}{8}\right)\right)\notag \\
&-\frac{1}{2}\underbrace{\sum_{n=0}^{\infty}\frac{2^n\Gamma\left(n+\frac{3}{2}\right)\Gamma\left(n+\frac{1}{2}\right)}{(2n+1)!}\left[\psi\left(n+\frac{3}{2}\right)-\psi\left(2n+2\right)\right]}_{Q}.
\end{align}

Applying (\ref{e:8}), the relation  $H_{n+1}=H_n+\frac{1}{n+1}$ and the duplication formulas  (\ref{e:5}) and (\ref{e:10}) for the gamma and digamma functions, this sum can be simplified  as follows:
\begin{align*}
&Q=\frac{\pi}{2}\sum_{n=0}^{\infty}\frac{1}{8^n}\binom{2n}{n}\left[H_{2n}-H_n-2\log 2+\frac{1}{2n+1}\right]
\end{align*}
or
\begin{align}\label{e:52}
&Q=-\pi\log 2\sum_{n=0}^{\infty}\frac{1}{8^n}\binom{2n}{n}+\frac{\pi}{2}\sum_{n=0}^{\infty}\frac{(1/8)^n}{2n+1}\binom{2n}{n}\notag\\
&+\frac{\pi}{2}\sum_{n=0}^{\infty}\binom{2n}{n}\frac{H_{2n}-H_n}{8^n}.
\end{align}
In \cite[Theorem 3]{10} it has been proved that
\begin{equation}\label{e:53}
\sum_{k=0}^{\infty}\binom{2k}{k}\left(H_{2k}-H_k\right)x^k=-\frac{1}{\sqrt{1-4x}}\log\left(\frac{1+\sqrt{1-4x}}{2}\right),\ (|x|<\frac{1}{4}).
\end{equation}
Using this formula for $x=1/8$ one gets
\begin{equation}\label{e:54}
\sum_{n=0}^{\infty}\binom{2n}{n}\frac{H_{2n}-H_n}{8^n}=-\sqrt{2}\log\left(\frac{2+\sqrt{2}}{4}\right).
\end{equation}
It is well known that
\begin{equation}\label{e:55}
\sum_{n=0}^{\infty}\frac{1}{8^n}\binom{2n}{n}=\sqrt{2}
\end{equation}
and
\begin{equation}\label{e:56}
\sum_{n=0}^{\infty}\frac{(1/8)^n}{2n+1}\binom{2n}{n}=\frac{\pi\sqrt{2}}{4}.
\end{equation}
Combining the identities  (\ref{e:52}), (\ref{e:54}), (\ref{e:55}) and (\ref{e:56}) we get
\begin{align}\label{e:57}
Q=-\frac{\pi ^2 \sqrt{2}}{16}+\frac{\pi  \sqrt{2}}{4}\log \left(2+\sqrt{2}\right).
\end{align}

We therefore conclude from (\ref{e:51}) and (\ref{e:57})  that
\begin{align*}
\int_{0}^{1}\frac{(x^2+1)\log\left(x^{2}+1\right)}{x^4+1}{\rm d }x&=\frac{1}{32} \left(\psi ^\prime\left(\frac{7}{8}\right)-\psi ^\prime\left(\frac{3}{8}\right)\right)\\
&-\frac{\pi ^2 \sqrt{2}}{16}+\frac{\pi  \sqrt{2}}{4}  \log \left(2+\sqrt{2}\right).
\end{align*}
\end{proof}
\begin{rem}
We note that a generalization of (\ref{e:53}) can be found in  \cite[Eq. (3.6)]{17}
\end{rem}

\begin{exam}\label{Example12} We have that
\begin{align*}
\sum_{n=0}^{\infty}\frac{2^n(H_{2n+1}-H_n)}{(2n+1)\binom{2n}{n}}=\frac{\pi}{4}\log 2+G.
\end{align*}
\end{exam}
\begin{proof} In \cite[Identity 4.7]{6} it has been showed that
\begin{equation*}
\sum_{k=0}^{n}\binom{2k}{k}\binom{2n-2k}{n-k}\frac{1}{(2k+1)^2}=\frac{16^n(H_{2n+1}-H_n)}{(2n+1)\binom{2n}{n}}.
\end{equation*}
Dividing both sides by $8^n$, summing over $n$, and using the Cauchy product formula for two convergent series, we get
\begin{equation*}
\bigg(\sum_{n=0}^{\infty}\binom{2n}{n}\frac{1}{8^n}\bigg)\bigg(\sum_{n=0}^{\infty}\binom{2n}{n}\frac{1}{8^n(2n+1)^2}\bigg)=\sum_{n=0}^{\infty}\frac{2^n(H_{2n+1}-H_n)}{(2n+1)\binom{2n}{n}}.
\end{equation*}
In \cite[Eq. 21]{8} it has been shown that
\begin{equation*}
\sum_{n=0}^{\infty}\binom{2n}{n}\frac{1}{8^n(2n+1)^2}=\frac{\sqrt{2}G}{2}+\frac{\sqrt{2}\pi\log 2}{8}.
\end{equation*}
Applying this and $\sum_{n=0}^{\infty}\binom{2n}{n}\frac{1}{8^n}=\sqrt{2}$ the conclusion follows.
\end{proof}
From \cite[Eq.(44)]{1} we have
\begin{align*}
\sum_{n=0}^{\infty}\frac{2^nH_n}{(2n+1)\binom{2n}{n}}=2G-\frac{\pi}{2}\log 2,
\end{align*}
so that, combining this with Example \ref{Example12} we get
\begin{exam}\label{Example13}
We have
\begin{align*}
\sum_{n=0}^{\infty}\frac{2^nH_{2n+1}}{(2n+1)\binom{2n}{n}}=3G-\frac{\pi}{4}\log 2.
\end{align*}
\end{exam}

\begin{exam}\label{Example14} We have the following formula
\begin{align*}
\int_{0}^{1}\frac{x^2\log\left(x^{2}+1\right)}{x^6+1}{\rm d }x=\frac{\pi}{6}  (\log 3-\log 2)-\frac{ 1}{9}G.
\end{align*}
\end{exam}
\begin{proof}
A simple change of variable leads to

\begin{align*}
2\int_{0}^{1}\frac{\log\left(x^{2/3}+1\right)}{x^2+1}{\rm d }x=6\int_{0}^{1}\frac{x^2\log\left(x^2+1\right)}{x^6+1}{\rm d}x.
\end{align*}
Setting $q=\frac{2}{3}$ in Theorem \ref{Theorem2.8} we get
\begin{align*}
2\int_{0}^{1}\frac{\log\left(x^{2/3}+1\right)}{x^2+1}{\rm d }x&=\frac{1}{24}\left(\psi^\prime(3/4)-\psi^\prime(1/4)\right)\\
&+\frac{3}{2}\sum_{n=0}^{\infty}\frac{3^n\Gamma\left(n+\frac{3}{2}\right)^2}{(2n+2)!}\left(\psi(2n+3)-\psi\left(n+\frac{3}{2}\right)\right).
\end{align*}
Making the change of variable $x=u^3$, using the duplication formulas for the gamma and digamma functions, and finally (\ref{e:5}) and $\psi^\prime(3/4)-\psi^\prime(1/4)=8G$ we are led to
\begin{align}\label{e:58}
&6\int_{0}^{1}\frac{x^2\log\left(x^2+1\right)}{x^6+1}{\rm d }x=-\frac{3\pi}{16}\sum_{n=0}^{\infty}\frac{(3/16)^n(2n+1)}{n+1}\binom{2n}{n}(H_{2n}-H_n)\notag\\
&+\frac{3\pi}{8}\log2\sum_{n=0}^{\infty}\frac{(3/16)^n(2n+1)}{n+1}\binom{2n}{n}-\frac{3\pi^2}{16}\sum_{n=0}^{\infty}\frac{(3/16)^n}{n+1}\binom{2n}{n}\notag\\
&+\frac{3\pi}{32}\sum_{n=0}^{\infty}\frac{(3/16)^n(2n+1)}{(n+1)^2}\binom{2n}{n}-\frac{2}{3}G.
\end{align}
We immediately have
\begin{align}\label{e:59}
&\sum_{n=0}^{\infty}\frac{(3/16)^n(2n+1)}{n+1}\binom{2n}{n}(H_{2n}-H_n)\notag\\
&=2\sum_{n=0}^{\infty}\left(\frac{3}{16}\right)^n\binom{2n}{n}(H_{2n}-H_n)-\sum_{n=0}^{\infty}\frac{(3/16)^n}{n+1}\binom{2n}{n}(H_{2n}-H_n).
\end{align}
Integrating both sides of (\ref{e:53})  from $x=0$  to $x=3/16$,  it follows that
\begin{equation*}
\sum_{n=0}^{\infty}\binom{2n}{n}\frac{(3/16)^n\left(H_{2n}-H_n\right)}{n+1}=\frac{16}{3}\int_{0}^{3/16}\frac{-1}{\sqrt{1-4x}}\log\left(\frac{1+\sqrt{1-4x}}{2}\right){\rm d}x.
\end{equation*}
The software \textit{Mathematica} can evaluate this integral and it gives us $\frac{1}{4}-\frac{3}{2}\log 2+\frac{3}{4}\log 3$. Hence,

\begin{equation*}
\sum_{n=0}^{\infty}\binom{2n}{n}\frac{(3/16)^n\left(H_{2n}-H_n\right)}{n+1}=\frac{4}{3}-8\log 2+4\log 3.
\end{equation*}
A direct computation of (\ref{e:53}) at $x=\frac{3}{16}$ leads to

\begin{equation}\label{e:60}
\sum_{n=0}^{\infty}\left(\frac{3}{16}\right)^n\binom{2n}{n}\left(H_{2n}-H_n\right)=4\log 2-2\log 3.
\end{equation}
Using the last two sums in (\ref{e:59}), we obtain

\begin{align}\label{e:61}
&\sum_{n=0}^{\infty}\frac{(3/16)^n(2n+1)}{n+1}\binom{2n}{n}(H_{2n}-H_n)=16 \log 2-8 \log 3-\frac{4}{3}.
\end{align}
Since
\begin{equation}\label{e:62}
\sum _{n=0}^{\infty } \left(\frac{3}{16}\right)^n \binom{2 n}{n}=2\quad \mbox{and}\quad \sum _{n=0}^{\infty } \frac{\left(\frac{3}{16}\right)^n \binom{2 n}{n}}{n+1}=\frac{4}{3}
\end{equation}
it follows that
\begin{align}\label{e:63}
\sum_{n=0}^{\infty}\frac{(3/16)^n(2n+1)}{n+1}\binom{2n}{n}=\frac{8}{3}.
\end{align}
\textit{Mathematica} gives

\begin{equation}\label{e:64}
\sum_{n=0}^{\infty}\frac{x^n}{n+1}\binom{2n}{n}=\frac{1-\sqrt{1-4 x}}{2 x}.
\end{equation}

We have
\begin{equation}\label{Eq4.17}
\sum_{n=0}^{\infty}\frac{(3/16)^n(2n+1)}{(n+1)^2}\binom{2n}{n}=2\sum_{n=0}^{\infty}\frac{(3/16)^n}{n+1}\binom{2n}{n}-\sum_{n=0}^{\infty}\frac{(3/16)^n}{(n+1)^2}\binom{2n}{n}.
\end{equation}
Integrating both sides of (\ref{e:64}) over $[0,3/16]$ we get
\begin{equation*}
\sum_{n=0}^{\infty}\frac{(3/16)^n}{(n+1)^2}\binom{2n}{n}=\frac{16}{3}\int_{0}^{3/16}\frac{1-\sqrt{1-4 x}}{2 x}{\rm d}x.
\end{equation*}
\textit{Mathematica} can evaluate this integral and
\begin{equation}\label{e:65}
\sum_{n=0}^{\infty}\frac{(3/16)^n}{(n+1)^2}\binom{2n}{n}=\frac{8}{3}+\frac{16}{3}\log 3-\frac{32}{3}\log 2,
\end{equation}
So it follows from (\ref{e:62}) and (\ref{Eq4.17}) that
\begin{equation}\label{e:66}
\sum_{n=0}^{\infty}\frac{(3/16)^n(2n+1)}{(n+1)^2}\binom{2n}{n}=\frac{32}{3}\log 2-\frac{16}{3}\log 3.
\end{equation}
Using the equations (\ref{e:61}), (\ref{e:62}), (\ref{e:63}) and (\ref{e:65}) in (\ref{e:58}) the conclusion follows.
\end{proof}
\begin{exam}\label{Example15} We have
\begin{align*}
\sum _{n=1}^{\infty } \frac{3^n \left(H_{2 n}-H_n\right)}{n \binom{2 n}{n}}=\sqrt{3} \left(Cl_2\left(\frac{2\pi}{3}\right)+\frac{\pi\log 3}{3}\right)-\frac{\pi ^2}{9},
\end{align*}
where $Cl_2(\varphi)=\sum_{n=1}^{\infty}\frac{\sin(\varphi n)}{n^2}$ is the Clausen's function of order 2.
\begin{proof}
We evaluate the formula \cite[Eq. (2-10)]{1} for $f(i)=\frac{1}{i}$. This gives
\begin{align}\label{e:67}
\sum _{n=1}^{\infty } \frac{3^n H_n}{n\binom{2n}{n}}=\sum_{n=1}^{\infty}\frac{3^n}{n^2\binom{2n}{n}}+\sqrt{3}\int_{0}^{3}\frac{2\arcsin(\sqrt{t}/2)}{4-t}{\rm d}t.
\end{align}
\textit{Mathematica} gives us

\begin{align}\label{e:68}
\sum_{n=1}^{\infty}\frac{3^n}{n^2\binom{2n}{n}}=\frac{2\pi^2}{9},
\end{align}
and
\begin{align}\label{e:69}
\int_{0}^{3}\frac{2\arcsin(\sqrt{t}/2)}{4-t}{\rm d }t=Cl_2\left(\frac{\pi}{3}\right).
\end{align}
Using (\ref{e:68}) and (\ref{e:69}) in (\ref{e:67}) we see that
\begin{align}\label{e:70}
\sum_{n=1}^{\infty}\frac{3^nH_n}{n\binom{2n}{n}}=\frac{2\pi^2}{9}+2\sqrt{3}Cl_2\left(\frac{\pi}{3}\right).
\end{align}
Applying \cite[Eq. (2-10)]{1} for $f(i)=\frac{1}{2i+1}$ and $x=3$, one can get

\begin{align}\label{e:71}
\sum_{k=1}^{\infty}\frac{3^k\sum_{i=1}^{k}\frac{1}{2i+1}}{k\binom{2k}{k}}&=\sum_{k=1}^{\infty}\frac{3^k}{k(2k+1)\binom{2k}{k}}\notag\\
&+\sqrt{3}\int_{0}^{3}\frac{\sum_{k=1}^{\infty}\frac{t^k}{(2k+1)\binom{2k}{k}}}{\sqrt{t(4-t)}}{\rm d}t.
\end{align}
\textit{Mathematica} gives
\begin{equation}\label{e:72}
\sum_{k=1}^{\infty}\frac{t^k}{(2k+1)\binom{2k}{k}}=\frac{4\sin^{-1}(\sqrt{t}/2)}{\sqrt{t(4-t)}}-1.
\end{equation}
and
\begin{align}\label{e:73}
\sum _{n=1}^{\infty } \frac{3^n}{n (2 n+1) \binom{2 n}{n}}=2-\frac{2\pi\sqrt{3}}{9}.
\end{align}
Using (\ref{e:72}) and (\ref{e:73}) in (\ref{e:71}) and noting that $\sum_{i=1}^{k}\frac{1}{2i+1}=H_{2k+1}-\frac{1}{2}H_k-1$, we can get, after simplifying
\begin{align}\label{e:74}
&\sum_{k=1}^{\infty}\frac{3^k\left(H_{2k}-\frac{1}{2}H_k\right)}{k\binom{2k}{k}}=\sum_{k=1}^{\infty}\frac{3^k}{k\binom{2k}{k}}\notag\\
&+\sqrt{3}\int_{0}^{x}\frac{4\sin^{-1}(\sqrt{t}/2)}{t(4-t)}-\sqrt{3}\int_{0}^{3}\frac{{\rm d}t}{\sqrt{t(4-t)}}.
\end{align}
Using
\begin{equation*}
\sum _{n=1}^{\infty } \frac{3^n}{n \binom{2 n}{n}}=\frac{2 \pi }{\sqrt{3}},
\end{equation*}
\begin{equation*}
\int_0^3 \frac{4 \sin ^{-1}\left(\frac{\sqrt{t}}{2}\right)}{t (4-t)}{\rm d }t=\frac{\pi}{3}\log 3+Cl_2\left(\frac{2\pi}{3}\right)+Cl_2\left(\frac{\pi}{3}\right)
\end{equation*}
and
\begin{equation*}
\int_0^3 \frac{1}{\sqrt{t (4-t)}} {\rm d }t=\frac{2 \pi }{3},
\end{equation*}
It follows from (\ref{e:74}) that

\begin{align}\label{e:75}
&\sum_{k=1}^{\infty}\frac{3^k\left(H_{2k}-\frac{1}{2}H_k\right)}{k\binom{2k}{k}}=\sqrt{3}\left(\frac{\pi}{3}\log 3+Cl_2\left(\frac{2\pi}{3}\right)+Cl_2\left(\frac{\pi}{3}\right)\right).
\end{align}
Thus, we conclude that
\begin{align*}
&\sum_{k=1}^{\infty}\frac{3^k\left(H_{2k}-H_k\right)}{k\binom{2k}{k}}=\sum_{k=1}^{\infty}\frac{3^k\left(H_{2k}-\frac{1}{2}H_k\right)}{k\binom{2k}{k}}\\
&-\frac{1}{2}\sum_{n=1}^{\infty}\frac{3^nH_n}{n\binom{2n}{n}}+\sqrt{3}\left(\frac{\pi}{3}\log 3+Cl_2\left(\frac{2\pi}{3}\right)+Cl_2\left(\frac{\pi}{3}\right)\right).
\end{align*}
Now the proof follows from (\ref{e:75}) and (\ref{e:70}).
\end{proof}
\end{exam}
\begin{exam}\label{Example16}
We have
\begin{align*}
\int_{0}^{1}\frac{(x+1)\log(x+1)}{x^3+1}{\rm d }x&=\frac{1}{36}\left(\psi^\prime(5/6)-\psi^\prime(1/3)\right)\\
&+\frac{\sqrt{3}}{3} \left(Cl_2\left(\frac{2\pi}{3}\right)+\frac{\pi\log 3}{3}\right)-\frac{\pi ^2}{27}.
\end{align*}
\end{exam}
\begin{proof}
Setting $q=1$ in Theorem \ref{Theorem2.8} we get
\begin{align*}
\int_{0}^{1}\frac{(x+1)\log(x+1)}{x^3+1}{\rm d }x&=\frac{1}{36}\left(\psi^\prime(5/6)-\psi^\prime(1/3)\right)\\
&+\frac{1}{3}\sum_{n=1}^{\infty}\frac{3^n\left(H_{2n}-H_n\right)}{n\binom{2n}{n}}.
\end{align*}
Applying Example \ref{Example15} this completes the proof.
\end{proof}
\begin{exam}\label{Example17}We have the following equality
\begin{align}\label{e:76}
\int_{0}^{1}\frac{\log\left(x^3+1\right)}{x+1}{\rm d }x&=\frac{1}{2}Li_2\left(-\frac{1}{3}\right)+\frac{\log^22}{2}+\frac{\log^23}{4}-\frac{\pi^2}{36}\\
&=\frac{1}{2}\log^22+\sum_{n=1}^{\infty}(-1)^{n+1}\sum_{k=2}^{n}\frac{(-1)^{k}}{k(k-1)\binom{n+k}{k}}, \label{e:77}
\end{align}
where $Li_2(z)=\sum_{n=1}^{\infty}\frac{z^n}{n^2}\quad (|z|\leq 1)$ is the dilogarithm function.
\end{exam}
\begin{proof}Using $x^3+1=(x+1)(1-x+x^2)$ we get
\begin{align}\label{e:78}
\int_{0}^{1}\frac{\log\left(x^3+1\right)}{x+1}{\rm d }x=\frac{1}{2}\log^22+\int_{0}^{1}\frac{\log\left(1-x+x^2\right)}{x+1}{\rm d }x.
\end{align}
On the other hand, expanding the  functions in the integrand into their power series it follows that
\begin{align*}
\int_{0}^{1}\frac{\log\left(1-x+x^2\right)}{x+1}{\rm d }x&=-\int_{0}^{1}\sum_{k=1}^{\infty}\frac{x^{k+1}(1-x)^{k-1}}{k+1}\cdot\sum_{k=0}^{\infty}(-1)^kx^kx\\
&=-\int_{0}^{1}\sum_{n=0}^{\infty}\sum_{k=0}^{n}\frac{(-1)^{n-k}x^{n+1}(1-x)^{k+1}}{k+1}{\rm d }x.
\end{align*}
Inverting the order of integration and summation we get
\begin{align*}
\int_{0}^{1}\frac{\log\left(1-x+x^2\right)}{x+1}{\rm d}x&=-\int_{0}^{1}\sum_{k=1}^{\infty}\frac{x^{k+1}(1-x)^{k+1}}{k+1}.\sum_{k=0}^{\infty}(-1)^kx^k{\rm d}x\\
&=-\int_{0}^{1}\sum_{n=0}^{\infty}\sum_{k=0}^{n}\frac{(-1)^{n-k}x^{n+1}(1-x)^{k+1}}{k+1}{\rm d}x\\
&=-\sum_{n=0}^{\infty}\sum_{k=0}^{n}\frac{(-1)^{n-k}}{k+1}\int_{0}^{1}x^{n+1}(1-x)^{k+1}{\rm d}x\\
&=-\sum_{n=0}^{\infty}\sum_{k=0}^{n}\frac{(-1)^{n-k}}{k+1}\int_{0}^{1}B(n+2,k+2){\rm d}x.
\end{align*}
Using the gamma-beta functional equation this gives

\begin{align*}
\int_{0}^{1}\frac{\log\left(1-x+x^2\right)}{x+1}{\rm d}x&=-\sum_{n=0}^{\infty}\sum_{k=0}^{n}\frac{(-1)^{n-k}}{k+1}\frac{(n+1)!(k+1)!}{(n+k+3)!}\\
&=\sum_{n=0}^{\infty}(-1)^n\sum_{k=0}^{n}\frac{(-1)^{k+1}}{(k+1)(n+k+3)\binom{n+k+2}{k+1}}\\
&=\sum_{n=0}^{\infty}(-1)^n\sum_{k=1}^{n}\frac{(-1)^{k}}{k(n+k+2)\binom{n+k+1}{k}}.
\end{align*}
Simplifying this gives
\begin{align}\label{e:79}
\int_{0}^{1}\frac{\log\left(1-x+x^2\right)}{x+1}{\rm d}x=\sum_{n=1}^{\infty}(-1)^{n+1}\sum_{k=2}^{n}\frac{(-1)^{k}}{k(k-1)\binom{n+k}{k}}.
\end{align}

Substituting (\ref{e:79}) in (\ref{e:78}) the proof of  (\ref{e:77}) completes. Now let us prove  (\ref{e:76}). If we make the substitution $x=\frac{1}{u}$, and use $x^3+1=(x+1)(1-x+x^2)$, we conclude that
\begin{align*}
\int_{0}^{1}\frac{\log\left(x^3+1\right)}{x+1}{\rm d}x&=\frac{\log^22}{2}+\int_1^{\infty } \frac{\log \left(t^2-t+1\right)-2 \log t}{t^2+t} {\rm d}t.
\end{align*}
Mathematica is able to evaluate the integral on the right-hand side and  gives
\begin{equation*}
\frac{1}{2}Li_2\left(-\frac{1}{3}\right)+\frac{\log^23}{4}-\frac{\pi^2}{36},
\end{equation*}
which completes the proof of (\ref{e:76}).
\end{proof}
\begin{exam}\label{Example18}We have
\begin{equation*}
\int_0^1 \frac{\log \left(x^3+1\right)}{x^2+1}{\rm d}x=-\frac{5}{3}G+\frac{\pi}{8}\log 2+\frac{\pi}{3}\log \left(\sqrt{3}+2\right)
\end{equation*}
\end{exam}
\begin{proof}
Making the substitution $x=\frac{1}{t}$, one easily gets
\begin{equation}\label{e:80}
\int_{0}^{1}\frac{\log\left(x^3+1\right)}{x^2+1}{\rm d}x=\frac{1}{2}\int_{0}^{\infty}\frac{\log\left(t^3+1\right)}{t^2+1}{\rm d}t-\frac{3}{2}\int_{1}^{\infty}\frac{\log x}{x^2+1}{\rm d}x
\end{equation}

\textit{Mathematica} gives us
\begin{equation}\label{e:81}
\int_0^{\infty } \frac{\log \left(x^3+1\right)}{x^2+1}{\rm d}x=-\frac{G}{3}+\frac{1}{4} \pi  \log (2)+\frac{1}{3} (2 \pi ) \log \left(\sqrt{3}+2\right)
\end{equation}
and
\begin{equation}\label{e:82}
\int_1^{\infty } \frac{\log (x)}{x^2+1}{\rm d}x=G.
\end{equation}
Combining the equations (\ref{e:80})- (\ref{e:82}) the conclusion immediately follows.
\end{proof}
\begin{rem}The integral in Example \ref{Example18} has been previously evaluated in \cite{19}.
\end{rem}
\begin{exam}\label{Example19}Setting $m=2$ in (\ref{e:45}), we have
\begin{equation*}
\frac{1}{\pi}\int_{0}^{\infty}\frac{\log\left(x^4+1\right)}{x^2+1}{\rm d}x=\log\left(2+\sqrt{2}\right).
\end{equation*}
\end{exam}
\begin{exam}\label{Example20}Setting $m=3$ in (\ref{e:46}), we have
\begin{equation*}
\frac{1}{\pi}\int_{0}^{\infty}\frac{\log\left(x^6+1\right)}{x^2+1}{\rm d}x=\log 6.
\end{equation*}
\end{exam}
\begin{exam}\label{Example21}Setting $m=4$ in (\ref{e:45}), we have
\begin{equation*}
\frac{1}{\pi}\int_{0}^{\infty}\frac{\log\left(x^8+1\right)}{x^2+1}{\rm d}x=\log \left(2+\sqrt{2-\sqrt{2}}\right)+\log \left(2+\sqrt{2+\sqrt{2}}\right).
\end{equation*}
\end{exam}
\begin{exam}\label{Example22}Setting $m=2$ in (\ref{e:47}), we have
\begin{align*}
\int_{0}^{\infty}\frac{x\log\left(x^6+1\right)}{x^3+1}{\rm d}x&=\frac{2 \pi ^2}{9}-\frac{2\pi\sqrt{3}}{9}  \log 2+\frac{2\pi\sqrt{3}}{9} \bigg[\log \left(\sqrt{3}+1\right)\\
&-2 \log \left(\sin \left(\frac{\pi }{12}\right)\right)\bigg].
\end{align*}
\end{exam}
\begin{exam}\label{Example23}Setting $m=4$ in (\ref{e:47}), we have
\begin{align*}
&\int_0^{\infty } \frac{x \log \left(x^{12}+1\right)}{x^3+1}{\rm d}x=\frac{4 \pi ^2}{9}-\frac{2 \pi  \sqrt{3}}{3} \log 2+\frac{2 \pi  \sqrt{3}}{9} \bigg[\log \left(\sqrt{2}+1\right)\\
&-2 \log \left(\sin \left(\frac{\pi }{8}\right)\right)-2 \log \left(\sin \left(\frac{\pi }{24}\right)\right)+\log \left(1+2 \cos \left(\frac{\pi }{12}\right)\right)\bigg].
\end{align*}
\end{exam}

\begin{exam}\label{Example24}Setting $m=3$ in (\ref{e:48}), we have
\begin{align*}
&\int_0^{\infty } \frac{x \log \left(x^9+1\right)}{x^3+1}{\rm d}x=\frac{\pi ^2}{3}-\frac{4 \pi  \sqrt{3}}{9} \log 2+\frac{5 \pi  \sqrt{3}}{9} \log 3\\
&-\frac{2 \pi  \sqrt{3}}{9}\bigg[2 \log \left(\sin \left(\frac{\pi }{9}\right)\right)-\log \left(1+2 \cos \left(\frac{2 \pi }{9}\right)\right)\bigg].
\end{align*}
\end{exam}

\begin{exam}\label{Example25}For  $m=3$, $m=5$ and $m=7$  Theorem \ref{Theorem3.1} produces, respectively
\begin{align*}
\sum_{n=1}^{\infty}(-1)^{n+1}\left(\frac{H_n}{6n+1}-\frac{H_n}{6n+3}+\frac{H_n}{6n+5}\right)=\frac{\pi}{2}\log 6-3G,
\end{align*}
\begin{align*}
\sum_{n=1}^{\infty}(-1)^{n+1}\bigg(\frac{H_n}{10n+1}&-\frac{H_n}{10n+3}+\frac{H_n}{10n+5}-\frac{H_n}{10n+7}+\frac{H_n}{10n+9}\bigg)\\
&=\frac{\pi}{2}\log\left(10+4\sqrt{5}\right)-5G.
\end{align*}
and
\allowdisplaybreaks
\begin{align*}
&\sum_{n=1}^{\infty}(-1)^{n+1}\bigg(\frac{H_n}{14n+1}-\frac{H_n}{14n+3}+\frac{H_n}{14n+5}-\frac{H_n}{14n+7}\\
&+\frac{H_n}{14n+9}-\frac{H_n}{14n+11}+\frac{H_n}{14n+13}\bigg)=2\pi\log 2-7G\\
&+\frac{\pi}{2}\bigg(\log\bigg(1+\sqrt{\frac{1}{2}\bigg(1-\cos\frac{\pi}{7}\bigg)}\bigg)+\log\bigg(1+\sqrt{\frac{1}{2}\bigg(1-\cos\frac{3\pi}{7}\bigg)}\bigg)\\
&+\log\bigg(1+\sqrt{\frac{1}{2}\bigg(1-\cos\frac{5\pi}{7}\bigg)}\bigg)\bigg).
\end{align*}
Here we used $\cos\frac{\pi}{5}=\frac{\sqrt{5}+1}{4}$ and $\cos\frac{3\pi}{5}=\frac{\sqrt{5}-1}{4}$.
\end{exam}

\begin{exam}\label{Example26} Applying Theorem \ref{Theorem3.2} for $m=1$  and $m=3$  we have, respectively

\begin{align*}
\sum_{n=1}^{\infty}(-1)^{n+1}\bigg(\frac{H_n}{3n+1}+\frac{H_n}{3n+2}\bigg)&=\frac{1}{12}\big(\psi^\prime(5/6)-\psi^\prime(1/3)\big)\\
&-\frac{\pi^2}{9}+\frac{\pi\sqrt{3}}{3}\log 3
\end{align*}
and
\begin{align*}
&\sum_{n=1}^{\infty}(-1)^{n+1}\bigg(\frac{H_n}{9n+1}+\frac{H_n}{9n+2}-\frac{H_n}{9n+4}-\frac{H_n}{9n+5}+\frac{H_n}{9n+7}
+\frac{H_n}{9n+8}\bigg)\\
&=\frac{1}{4}\left(\psi^\prime(5/6)-\psi^\prime(1/3)\right)-\frac{\pi^2}{3}+\frac{5\pi\sqrt{3}}{9}\log 3-\frac{4\pi\sqrt{3}}{9} \log2\\
&-\frac{2\pi\sqrt{3}}{9}\left[2\log\left(\sin\frac{\pi}{9}\right)-\log\left(2\cos\frac{2\pi}{9}+1\right)\right].
\end{align*}
\end{exam}
\begin{exam}\label{Example27} Applying Theorem \ref{Theorem3.3} for $m=1$  and $m=3$  we have, respectively

\begin{align*}
\sum_{n=1}^{\infty}(-1)^{n+1}\bigg(\frac{H_n}{4n+1}+\frac{H_n}{4n+2}\bigg)&=\frac{1}{16}\big(\psi^\prime(7/8)-\psi^\prime(3/8)\big)\\
&+\frac{3\pi\sqrt{2}}{4}\log 2-\frac{\pi^2\sqrt{2}}{8}
\end{align*}
and

\begin{align*}
&\sum_{n=1}^{\infty}(-1)^{n+1}\bigg(\frac{H_n}{12n+1}+\frac{H_n}{12n+3}-\frac{H_n}{12n+5}-\frac{H_n}{12n+7}+\frac{H_n}{12n+9}\\
&+\frac{H_n}{12n+11}\bigg)=\frac{1}{8}\left(\psi^\prime(7/8)-\psi^\prime(3/8)\right)+\frac{3\pi\sqrt{2}}{4}\log 2-\frac{3\pi^2\sqrt{2}}{8}\\
&+\frac{\pi\sqrt{2}}{4}\log\bigg(21+12\sqrt{6}\bigg).
\end{align*}

\end{exam}
\section{Concluding remarks}
\begin{rem}
We verified all of our examples numerically one by one through  \textit{Mathematica} program to assure their accuracy.
\end{rem}
\begin{rem}
Examples \ref{Example1}, \ref{Example2} and \ref{Example3}  are not new and special cases of some general formulas proved in \cite{19}. Example \ref{Example23} is also not new and  proved in \cite{19}. It seems that the rest of our examples has not appeared previously in the literature. Example \ref{Example8} appeared as monthly problem 12221 proposed by the author in \cite{5}. The computation of the value of the integral given in Example \ref{Example20} has been proposed as a magazine problem 2107 in \cite{25} by S. M. Stewart.
\end{rem}
\begin{rem}
Employing the results obtained in \cite{14}, \cite{11} and \cite{15} we can evaluate the values of the polygamma function $\psi^\prime(x)$ at rational arguments in terms of some well known special functions such as the Riemann zeta function, Clausen's function and the Hurwitz zeta function.
\end{rem}
\begin{rem}It is important to note that using the pochhammer symbol $(a)_k=\frac{\Gamma(a+k)}{\Gamma(a)}$, the reflection formula for the gamma function and the simple fact $(2k)!=2^{2k}k!(\frac{1}{2})_k$,  the series given in the statements of Theorem \ref{Theorem1} and Corollary \ref{Corollary3} can be written in terms of the hypergeometric function ${}_2F_1(a,b;c:x)$ respectively as follows:
\begin{equation*}
\sum_{k=1}^{\infty}\frac{2^k}{k}\frac{\Gamma\left(k+\frac{1}{q}\right)\Gamma\left(k+1-\frac{1}{q}\right)}{(2k)!}=\frac{\pi}{\sin\frac{\pi}{q}}\int_{0}^{1/2}\frac{{}_2F_1\left(\frac{1}{q},1-\frac{1}{q};\frac{1}{2}:x\right){\rm d}x}{x}
\end{equation*}
and
\begin{equation*}
\sum_{k=1}^{\infty}\frac{3^k}{k}\frac{\Gamma\left(k+\frac{1}{q}\right)\Gamma\left(k+1-\frac{1}{q}\right)}{(2k)!}=\frac{\pi}{\sin\frac{\pi}{q}}\int_{0}^{3/4}\frac{{}_2F_1\left(\frac{1}{q},1-\frac{1}{q};\frac{1}{2}:x\right){\rm d}x}{x},
\end{equation*}
with $q>1$.
\end{rem}
\begin{rem} For all even integers $q\geq 2$ the integrals
\begin{equation*}
\int_{0}^{1}\frac{\log\left(x^q+1\right)}{x+1}{\rm d}x
\end{equation*}
have been calculated in \cite{23}; see also \cite{21}. For $q=3$ we in our Example~\ref{Example17} calculated it but for odd integers  $q\geq 5$ the value of these integrals are not known in  the literature as far as we know, thus, it will be an interesting area of  source to study these integrals for odd integers $q\geq 5$. These integrals are closely connected with the Herglotz function $F$, defined, for $q\in\mathbb{C}\backslash(-\infty,0]$, by
$$
F(q)=\sum_{n=1}^{\infty}\frac{1}{n}(\psi(qn)-\log(qn)),
$$
where $\psi$ is the digamma function. Indeed, we have
\begin{equation*}\label{Eq. 1}
\int_{0}^{1}\frac{\log\left(x^q+1\right)}{x+1}{\rm d}x=F(2q)-2F(q)+F(q/2)+\frac{1}{2q}\zeta(2);
\end{equation*}
see \cite{16,18,21}.
\end{rem}
\begin{rem} For even integers $m\geq 2$ we in our Theorem~\ref{Theorem2.4} evaluate the family of integrals
\begin{equation*}
\int_{0}^{1}\frac{\log\left(x^m+1\right)}{x^2+1}{\rm d}x
\end{equation*}
but in the case of $m$ odd  we are unable to evaluate them except for $m=3$. Thus, the evaluation of these integrals for odd integer $m\geq 5$  is an interesting subject of investigations.
\end{rem}
\textbf{Acknowledgements} The author thanks the editor and the anonymous referee for their useful suggestions which improve the presentation of the paper.
\section{Funding} No funding was received to assist with the preparation of this manuscript.
\section{Declarations}

\textbf{Conflict of interest} The author has no conflicts of interest to declare that are
relevant to the content of this article.
\section{Data availability}
This paper has no associated data.

 \end{document}